\RequirePackage{fix-cm}
\documentclass[smallcondensed]{svjour3}     
\smartqed
\usepackage{graphicx,subfig}
\usepackage{amsmath,amssymb,amsfonts}
\usepackage{bm}
\usepackage{algorithm, algorithmic}
\usepackage{multirow}
\usepackage{chngcntr}
\newcommand\tOmega{\widetilde{\Omega}}
\newcommand\tGamma{\widetilde{\Gamma}}
\newcommand\tV{\widetilde{V}}
\newcommand\W{\mathbb{W}}
\newcommand\tu{\tilde{u}}
\newcommand\bu{\bar{u}}
\newcommand\hu{\hat{u}}
\newcommand\bl{\bar{\lambda}}
\newcommand\p{\mathbf{p}}
\renewcommand\P{\mathbf{P}}
\newcommand\intO{\int_{\Omega}}
\newcommand\intOs{\int_{\Omega_s}}

\newcommand\tE{\widetilde{E}}
\newcommand\N{\mathcal{N}}
\newcommand\TOL{\mathrm{TOL}}

\let\div\relax
\DeclareMathOperator\div{div}
\DeclareMathOperator\proj{proj}

\DeclareMathOperator\ed{ed}
\DeclareMathOperator\supp{supp}
\DeclareMathOperator*{\argmin}{\arg\min}

\numberwithin{equation}{section}

\counterwithin{theorem}{section}

\spnewtheorem{proposition}[theorem]{Proposition}{\bfseries}{\itshape}
\spnewtheorem{lemma}[theorem]{Lemma}{\bfseries}{\itshape}
\spnewtheorem{corollary}[theorem]{Corollary}{\bfseries}{\itshape}
\spnewtheorem{remark}[theorem]{Remark}{\itshape}{\rmfamily}
\spnewtheorem{definition}[theorem]{Definition}{\bfseries}{\rmfamily}

\usepackage{algorithm, algorithmic}

\journalname{arXiv}

\begin{document}

\title{An Overlapping Domain Decomposition Framework without Dual Formulation for Variational Imaging Problems}
\titlerunning{Overlapping Domain Decomposition for Imaging}
\author{Jongho Park}
\institute{Jongho Park \at
              Department of Mathematical Sciences, KAIST, Daejeon 34141, Korea \\
              \email{jongho.park@kaist.ac.kr}
}

\date{Received: date / Accepted: date}

\maketitle

\begin{abstract}
In this paper, we propose a novel overlapping domain decomposition method that can be applied to various problems in variational imaging such as total variation minimization.
Most of recent domain decomposition methods for total variation minimization adopt the Fenchel--Rockafellar duality, whereas the proposed method is based on the primal formulation.
Thus, the proposed method can be applied not only to total variation minimization but also to those with complex dual problems such as higher order models.
In the proposed method, an equivalent formulation of the model problem with parallel structure is constructed using a custom overlapping domain decomposition scheme with the notion of essential domains.
As a solver for the constructed formulation, we propose a decoupled augmented Lagrangian method for untying the coupling of adjacent subdomains.
Convergence analysis of the decoupled augmented Lagrangian method is provided.
We present implementation details and numerical examples for various model problems including total variation minimizations and higher order models.
\end{abstract}

\section{Introduction}
\label{Sec:Introduction}
Most problems of the variational approach to image processing have the form of
\begin{equation}
\label{general}
\min_{u} \left\{ E (u) := F(Au) + R(u) \right\}.
\end{equation}
Here, $F(Au)$ is a fidelity term which measures a distance between the given image $f$ and a solution $u$.
The linear operator $A$ is determined by the type of the problem~\cite{CW:1998,ROF:1992,SC:2002}.
We use $A = I$ for the image denoising problem~\cite{ROF:1992}.
Image denoising problems are converted to image inpainting problems when we set $A$ by the restriction operator to the subset corresponding to the known part of the image~\cite{SC:2002}.
For the image deconvolution problem, $A$ models a blur kernel~\cite{CW:1998}.
 The functional $F$ is usually given by a norm of the difference between $Au$ and $f$.
 In the case of image denoising, $L^2$-norm is adopted to catch Gaussian noise~\cite{ROF:1992}, while $L^1$-norm is used when the image is corrupted by impulse noise~\cite{CE:2005,Nikolova:2004}.
In addition, there are variational denoising models with specific norms to treat various types of noise, for example, see~\cite{AA:2008,LCA:2007}.  
 
 On the other hand, $R(u)$ plays a role of a regularizer which resolves the ill-posedness of the problem and enforces the regularity of the solution.
 The most primitive is the $H^1$-regularization proposed by Tikhonov~\cite{Tikhonov:1963}, where $R(u)$ is given by the $H^1$-seminorm of $u$.
 To preserve edges or discontinuities of the image, a class of nonsmooth regularizers has been considered.
 The famous Rudin--Osher--Fatemi~(ROF) model which uses total variation as a regularizer for image denoising was proposed in~\cite{ROF:1992}, and it successfully removes Gaussian noise while it preserves edges of the image.
Since a solution of the total variation minimization problem is piecewise constant in general, it causes the staircase effect on the resulting image.
To avoid such a situation, higher order regularizers which are expressed in terms of higher order derivatives of $u$ have been proposed in numerous literature~\cite{CL:1997,LLT:2003}.

The imaging problems introduced above are nonseparable in general; see~\cite[Assumption~3.1]{LP:2019} for the definition of the nonseparability.
To be more precise, suppose that the image domain $\Omega$ is partitioned into nonoverlapping subdomains $\{ \Omega_s \}$.
Then the energy functional of the imaging problem defined on the full domain $\Omega$ cannot be expressed as the sum of local energy functionals defined on subdomains $\Omega_s$, i.e., there do not exist local energy functionals $\{ E_s \}$ such that
\begin{equation*}
E(u) = \sum_{s=1}^{\N} E_s (u|_{\Omega_s}).
\end{equation*}
For example, the total variation regularizer proposed in~\cite{ROF:1992} is nonseparable since it measures the jump of a function across the subdomain interfaces.
On the other hand, fidelity terms of image deconvolution problems are nonseparable due to the nonlocal nature of the convolution.
Due to the nonseparability, it has been considered as a difficult problem to design efficient block methods or domain decomposition methods~(DDMs) for imaging problems.
Indeed, it was shown in~\cite{LN:2017} that the usual block relaxation methods such as Jacobi and Gauss--Seidel applied to the ROF model are not guaranteed to converge to a minimizer.

The purpose of this paper is to introduce a novel convergent DDM for a family of problems of the form~\eqref{general}.
In DDMs, the domain of the problem is decomposed into either overlapping or nonoverlapping subdomains.
Then, we decompose the full-dimension problem into smaller dimension problems on subdomains, called local problems.
Since the local problems can be solved in parallel, DDMs implemented on distributed memory computers are efficient ways to treat large scale images.

There have been numerous researches on DDMs for a particular case of~\eqref{general}: total variation minimization.
Subspace correction methods for the ROF model
\begin{equation}
\label{ROF_intro}
\min_{u \in BV(\Omega)} \frac{\alpha}{2} \intO (u-f)^2 \,dx + TV(u)
\end{equation}
were considered in several papers; see, e.g.,~\cite{FS:2009}, where $BV(\Omega)$ is a set of functions with bounded variation in $\Omega$ and $TV(u)$ is the total variation of $u$.
However, as we mentioned above, those methods may converge to a wrong solution due to the nonseparability of $TV(u)$~\cite{LN:2017}.
To overcome such difficulties, a number of recent papers~\cite{CTWY:2015,LPP:2019} dealt with the Fenchel--Rockafellar dual problem of~\eqref{ROF_intro} given by
\begin{equation}
\label{dual_ROF_intro}
\min_{\p \in (C_0^1 (\Omega))^2} \frac{1}{2\alpha} \intO (\div \p + \alpha f )^2 \,dx \quad \textrm{ subject to } \| \p \|_{\infty} \leq 1
\end{equation}
instead of the original~(primal) one.
In the case of~\eqref{dual_ROF_intro}, the constraint $\| \p \|_{\infty} \leq 1$ can be treated separately in each subdomain.
Moreover, the solution space $(C_0^1 (\Omega))^2$ has good regularity so that it is able to impose appropriate boundary conditions to local problems in subdomains.
With these advantages, iterative substructuring methods for~\eqref{dual_ROF_intro} were proposed in~\cite{LPP:2019}, and one of them was generalized for general total variation minimization in~\cite{LP:2019}.

However, the dual approach is not adequate to apply to the general variational problem~\eqref{general}.
First, it is hard to obtain an explicit formula for the Fenchel--Rockafellar dual formulation of~\eqref{general}.
There are researches on the dual formulations for particular cases of~\eqref{general}; see~\cite{Chambolle:2004,DHN:2009} for instance.
We also note that the duality based DDM proposed in~\cite{LP:2019} cannot be applied to problems with nonseparable fidelity terms like image deconvolution problems even if their energy functionals are convex.

In this paper, we propose an overlapping domain decomposition framework that does not rely on the Fenchel--Rockafellar duality.
The proposed framework has very wide range of applications.
It accommodates almost all total variation-regularized problems containing ones with nonseparable fidelity terms.
With a little modification, it can be applied to problems with higher order regularizers such as~\cite{LLT:2003}. 

The proposed framework is constructed by using the notion of \textit{essential domains}.
First, the image domain $\Omega$ is partitioned into nonoverlapping subdomains $\left\{ \Omega_s \right\}$.
For each $\Omega_s$, there exists a slightly larger subset $\tilde{\Omega}_s$ of $\Omega$ such that the computation of the local energy functional $E_s (u)$ on $\Omega_s$ requires only the values of $u$ on $\tilde{\Omega}_s$.
We call the minimal $\tilde{\Omega}_s$ the essential domain for $E_s$.
Then, $\{ \tilde{\Omega}_s \}$ forms an overlapping domain decomposition of $\Omega$ and we can construct an equivalent constrained minimization problem with a parallel structure.
The proposed formulation can be regarded as a generalization of~\cite{DCT:2016} in the sense that we get exactly the same formulation as in~\cite{DCT:2016} if we apply the proposed framework to the convex Chan--Vese model~\cite{CV:2001} for image segmentation.

While \cite{DCT:2016} adopts the first order primal-dual algorithm~\cite{CP:2011} to solve the resulting equivalent minimization problem, we use a version of the augmented Lagrangian method~\cite{Hestenes:1969}.
As it is well-known that the penalty term appearing in the augmented Lagrangian method couples local problems in adjacent subdomains~\cite{LP:2009}, we propose a \textit{decoupled augmented Lagrangian method} which guarantees parallel computation of local problems.
Differently from the conventional augmented Lagrangian method, a modified penalty term which does not couple adjacent local problems is used in the decoupled augmented Lagrangian method.
Convergence analysis of the proposed method can be done in a similar way as the conventional analysis for the alternating direction method of multipliers given in~\cite{HY:2015,WT:2010}.
Numerical experiments ensure that the proposed method outperforms the existing methods~\cite{DCT:2016,LNP:2019} for various imaging problems.

We summarize the main advantages of this paper in the following.
\begin{itemize}
\item Since the proposed method does not utilize Fenchel--Rockafellar duality, it has wide range of applications; it can be applied to problems with complex dual problems.
\item The proposed method is suitable to implement on distributed memory computers.
Moreover, it is easy to program the proposed method since there is no data structures lying on the subdomain interfaces, which make parallel implementation hard.
\item While almost of the existing works with convergence guarantee use the dual approach, with the novel overlapping domain decomposition framework using essential domains, convergence to a correct minimizer is guaranteed without the dual approach.
\end{itemize}

The rest of the paper is organized as follows. 
In Section~\ref{Sec:DD}, we introduce the notion of essential domains and propose an overlapping domain decomposition framework.
A decoupled augmented Lagrangian method to solve the domain decomposition formulation is presented in Section~\ref{Sec:ALM}.
We apply the proposed DDM to several variational models in image processing, including total variation minimizations and higher order models in Sections~\ref{Sec:Applications}.
We conclude the paper with remarks in Section~\ref{Sec:Conclusion}.

\section{Domain decomposition framework}
\label{Sec:DD}
In this section, we briefly state a discrete setting for~\eqref{general} first.
Then, an overlapping domain decomposition framework using the notion of essential domains is introduced.
Throughout this paper, for the generic $n$-dimensional Hilbert space $H$ and $1 \leq p < \infty$, the $p$-norm of $v \in H$ is denoted by $\| v \|_{p, H}$ and the Euclidean inner product of $v, w \in H$ by $\langle v, w \rangle_H$.
The subscript $H$ can be deleted if there is no ambiguity.
The dual space of $H$ consisting of all linear functionals on $H$ is denoted by $H^*$.
Let $J_H$:~$H \rightarrow H^*$ be the Riesz isomorphism from $H$ to $H^*$, i.e.,
\begin{equation*}
(J_H u)(v) = \langle u, v \rangle_H , \hspace{0,5cm} u, v \in H.
\end{equation*}
By identifying $H$ with its double dual $H^{**}$, we have $J_{H^*} = (J_H)^*$ and $J_{H^*} J_H = I$.

Consider a grayscale image of the resolution $M \times N$.
We regard each pixel in the image as a discrete point, i.e., the image domain~$\Omega$ consists of $M \times N$ discrete points.
Let $V$ be the collection of all functions from $\Omega$ to $\mathbb{R}$.
In this setting, the discrete integration of $v \in V$ is naturally evaulated as
\begin{equation*}
\intO v \, dx = \sum_{(i,j) \in \Omega} v_{ij}.
\end{equation*}
The linear operator $A$ in~\eqref{general} can be regarded as a linear operator on $V$.

We assume that the energy functional in~\eqref{general} has the discrete integral structure, that is, there exists an operator $T$:~$V \rightarrow V$ such that
\begin{equation}
\label{integral}
E (u) = \intO T (u) \, dx.
\end{equation}
This assumption is reasonable since most of popular variational models in image processing are of this form.
For example, for the discrete ROF model introduced in~\cite{Chambolle:2004}, we have
\begin{equation*}
T (u) = \frac{\alpha}{2} (u - f) ^2 + | \nabla^+ u |,
\end{equation*}
where $\alpha$ is a positive parameter and $\nabla^+$ is the discrete gradient operator which will be defined rigorously in Section~\ref{Sec:Applications}.

First, we consider a nonoverlapping domain decomposition of $\Omega$.
The image domain $\Omega$ is decomposed into $\N$ disjoint rectangular subdomains $\left\{ \Omega_s \right\}_{s=1}^{\N}$.
One can consider the (nonoverlapping) local function space $V_s$ on $\Omega_s$ as
\begin{equation}
\label{V_s}
V_s = \left\{ v \in V : \supp v \subset \Omega_s \right\}.
\end{equation}
To construct an overlapping domain decomposition which is suitable for parallel computation, we introduce the notion of essential domains.

\begin{definition}
Let $D \subset \Omega$ and $T$: $V \rightarrow V$.
The \emph{essential domain} of $T$ on $D$, denoted by $\ed_D (T)$, is defined as the minimal subset $\tilde{D}$ of $\Omega$ such that
$T(u)|_{D}$ can be expressed with the values of $u|_{\tilde{D}}$ only.
\end{definition}
\noindent
Examples of essential domain will be presented in Section~\ref{Sec:Applications}. 

We define the local energy functionals $E_s$: $V \rightarrow \mathbb{R}$ as
\begin{equation*}
E_s (u) = \intOs T (u) \, dx.
\end{equation*}
In the computation of $E_s (u)$, we need the values of $u$ not on the entire domain $\Omega$ but on $\ed_{\Omega_s} (T)$.
We set
\begin{equation*}
\tOmega_s = \ed_{\Omega_s} (T).
\end{equation*}
Then, $\{ \tOmega_s \}_{s=1}^{\N}$ forms an overlapping domain decomposition of $\Omega$.
Since $E_s (u) = E_s (u|_{\tOmega_s})$, we have
\begin{equation*}
E (u) = \sum_{s=1}^{\N} E_s (u) = \sum_{s=1}^{\N} E_s (u|_{\tOmega_s}).
\end{equation*}

The (overlapping) local function space $\tV_s$ for the subdomain $\tOmega_s$ is defined as
\begin{equation*}
\tV_s = \left\{ v \in V : \supp v \subset \tOmega_s \right\},
\end{equation*}
and
\begin{equation*}
\tV = \bigoplus_{s=1}^{\N} \tV_s.
\end{equation*}
It is easy to observe that $V_s \subset \tV_s$.
Also, we define the thick interface $\tGamma_{st} = \tOmega_s \cap \tOmega_t$ for $s<t$ and $\tGamma = \bigoplus_{s<t} \tGamma_{st}$.
Let $\tV_{\tGamma}$ be the collection of all functions from $\tGamma$ to $\mathbb{R}$.
Take $\tu_s \in \tV_s$ for all $s$ and let $\tu = \bigoplus_{s=1}^{\N} \tu_s \in \tV$.
If $\tu_s = \tu_t$ on $\tGamma_{st}$ for all $s<t$, then $\tu$ can be considered as an element of $V$, i.e., $\tu \in V$.
In this case, the ``splitted'' energy functional $\tE$:~$\tV \rightarrow \mathbb{R}$ defined by
\begin{equation*}
\tE (\tu) = \sum_{s=1}^{\N} E_s (\tu_s), \hspace{0.5cm} \tu = \bigoplus_{s=1}^{\N} \tu_s \in \tV
\end{equation*}
agrees with $E (\tu)$.
It motivates the jump operator $B$: $\tV \rightarrow \tV_{\tGamma}^*$ to be defined as
\begin{equation*}
B\tu |_{\tGamma_{st}} := \tu_s |_{\tGamma_{st}} - \tu_t |_{\tGamma_{st}}, \hspace{0.5cm} s<t.
\end{equation*}
Then we conclude that minimizing $E$ over $V$ is equivalent to minimizing $\tE$ over $\ker B \subset \tV$.
We summarize this fact in the following theorem.

\begin{theorem}
\label{Thm:DD_equiv}
Let $\tu^* \in \tV$ be a solution of the constrained minimization problem
\begin{equation}
\label{DD}
\min_{\tu \in \tV} \tE(\tu) \hspace{0.5cm} \textrm{subject to } B\tu = 0.
\end{equation}
Then, we have $\tu^* \in V$, which is a solution of the minimization problem
\begin{equation*}
\min_{u \in V} E(u).
\end{equation*}
\end{theorem}

\section{Decoupled augmented Lagrangian method}
\label{Sec:ALM}
In this section, we discuss how to design a subdomain-level parallel algorithm to solve~\eqref{DD}.
First, the augmented Lagrangian formulation can be considered to handle the constraint $B \tu = 0$ in~\eqref{DD}:
\begin{equation}
\label{ALM_old}
\min_{\tu \in \tV} \max_{\mu \in \tV_{\tGamma}^*} \left\{ \tE(\tu) + \langle B\tu, \mu \rangle_{\tV_{\tGamma}^*} + \frac{\eta}{2} \| B \tu \|_{2, \tV_{\tGamma}^*}^2  \right\},
\end{equation}
where $\mu \in \tV_{\tGamma}^*$ is a Lagrange multiplier and $\eta >0$ is a penalty parameter.
The $\tu$-subproblem in the augmented Lagrangian formulation~\eqref{ALM_old} reads as follows:
\begin{equation}
\label{tu1}
\tu^{(n+1)} \in \argmin_{\tu \in \tV} \left\{ \tE (\tu) +\langle B\tu, \mu^{(n)} \rangle_{\tV_{\tGamma}^*} + \frac{\eta}{2} \| B \tu \|_{2, \tV_{\tGamma}^*}^2 \right\}.
\end{equation}
In~\eqref{tu1}, due to the penalty term $\frac{\eta}{2} \| B \tu \|_{2, \tV_{\tGamma}^*}^2$, local problems in subdomains are coupled so that they cannot be solved independently.
Therefore, \eqref{tu1} is not adequate for subdomain-level parallel computation.
Such a phenomenon was previously observed in~\cite{LP:2009}.
To resolve this difficulty, first, we replace~\eqref{DD}  by an equivalent one:
\begin{equation}
\label{DD_modified}
\min_{\tu \in \tV} \tE(\tu) \hspace{0.5cm} \textrm{subject to } (I-P_B)\tu = 0,
\end{equation}
where $P_B$: $\tV \rightarrow \tV$ is the orthogonal projection onto $\ker B$.
We note that computation of $P_B$ does not require explicit assembly of the matrix for $P_B$.
One can easily check that if $(i, j) \in \Omega$ is shared by $k$ (overlapping) subdomains, then $(P_B \tu)_{ij}$ is the average of $(\tu_s)_{ij}$ in the $k$ overlapping subdomains.
The $\tu$-subproblem in the augmented Lagrangian method for~\eqref{DD_modified} is the same as~\eqref{tu1} except that $B\tu \in \tV_{\tGamma}^*$ is replaced by $J_{\tV} (I-P_B) \tu \in \tV^*$:
\begin{equation}
\label{tu2}
\tu^{(n+1)} \in \argmin_{\tu \in \tV} \left\{ \tE (\tu) + \langle J_{\tV} (I - P_B) \tu, \lambda^{(n)} \rangle_{\tV^*} + \frac{\eta}{2} \| (I -P_B) \tu \|_{2, \tV}^2 \right\}.
\end{equation}
Note that the Lagrange multiplier $\lambda^{(n)}$ in~\eqref{tu2} is in $\tV^*$, while $\mu^{(n)}$ in~\eqref{tu1} is in $\tV_{\tGamma}^*$.
Next, we replace $(I - P_B) \tu$ in the penalty term in~\eqref{tu2} by $\tu - P_B \tu^{(n)}$:
\begin{equation}
\label{tu3}
\tu^{(n+1)} \in \argmin_{\tu \in \tV} \left\{ \tE (\tu) + \langle J_{\tV} (I - P_B) \tu, \lambda^{(n)} \rangle_{\tV^*} + \frac{\eta}{2} \| \tu - P_B \tu^{(n)} \|_{2, \tV}^2 \right\}.
\end{equation}
To further simplify the resulting algorithm, we may discard $I-P_B$ in the inner product term in~\eqref{tu3} with the assumption that $J_{\tV^*} \lambda^{(n)} \in  (\ker B)^{\bot}$ for all $n \geq 0$.
We will see in Proposition~\ref{Prop:kerB} that such an assumption is convincing.
Now, we have
\begin{equation}
\label{tu_final}
\tu^{(n+1)} \in \argmin_{\tu \in \tV} \left\{ \tE (\tu) + \langle J_{\tV} \tu, \lambda^{(n)} \rangle_{\tV^*} + \frac{\eta}{2} \| \tu - P_B \tu^{(n)} \|_{2, \tV}^2 \right\}.
\end{equation}
Then, the local problems of~\eqref{tu_final} are decoupled in the sense that a solution of~\eqref{tu_final} is obtained by assembling the solutions of $\N$ independent local problems in the subdomains.
Indeed, we have $\tu^{(n+1)} = \bigoplus_{s=1}^{\N} \tu_s^{(n+1)}$, where
\begin{equation}
\label{local}
\tu_s^{(n+1)} \in \argmin_{\tu_s \in \tV_s} \left\{ E_s (\tu_s ) + \langle J_{\tV_s} \tu_s , \lambda_s^{(n)} \rangle_{\tV_s ^*} + \frac{\eta}{2} \| \tu_s - (P_B \tu^{(n)})_s \|_{2, \tV_s}^2 \right\} .
\end{equation}
For the sake of convenience,~\eqref{local} is rewritten as
\begin{equation}
\label{local2}
\tu_s^{(n+1)} \in \argmin_{\tu_s \tV_s} \left\{ E_s (\tu_s ) + \frac{\eta}{2} \| \tu_s - \hu_s^{(n+1)} \|_{2, \tV_s}^2 \right\},
\end{equation}
where
\begin{equation*}
\hu_s^{(n+1)} = (P_B \tu^{(n)})_s - \frac{J_{\tV_s^*} \lambda_s^{(n)}}{\eta}.
\end{equation*}
In summary, we propose a decoupled augmented Lagrangian method for~\eqref{DD_modified} in Algorithm~\ref{Alg:DD}.

\begin{algorithm}[]
\caption{Decoupled augmented Lagrangian method for \eqref{DD_modified}}
\begin{algorithmic}[]
\label{Alg:DD}
\STATE Choose $\eta > 0$.
Let $\tu^{(0)} \in \tV$ and $J_{\tV^*} \lambda^{(0)} \in (\ker B)^{\bot}$.
\FOR{$n=0,1,2,\dots$}
\FOR{$s=1,\dots, \N$ \textbf{ in parallel}}
\STATE $\displaystyle \hu_s^{(n+1)} = (P_B \tu^{(n)})_s - \frac{J_{\tV_s^*} \lambda_s^{(n)}}{\eta} $
\STATE $\displaystyle \tu_s^{(n+1)} \in \argmin_{\tu_s \in \tV_s} \left\{ E_s (\tu_s ) + \frac{\eta}{2} \| \tu_s - \hu_s^{(n+1)} \|_{2, \tV_s}^2 \right\}$
\ENDFOR
\STATE $\displaystyle \tu^{(n+1)} = \bigoplus_{s=1}^{\N} \tu_s^{(n+1)}$
\STATE $\lambda^{(n+1)} = \lambda^{(n)} + \eta J_{\tV} (I - P_B ) \tu^{(n+1)}$
\ENDFOR
\end{algorithmic}
\end{algorithm}

In Algorithm~\ref{Alg:DD}, the only step that requires communication among subdomains is the computation of $P_B$.
As we noticed above, implementation of $P_B$ is easy because it is simple pointwise averaging.
If $(i,j) \in \Omega$ is shared by $k$ subdomains, then addition of $k$ scalars followed by division by $k$ is required to compute $P_B$ at $(i,j)$.
In addition, data communication among $k$ subdomains is needed.
All the other steps of Algorithm~\ref{Alg:DD} can be done independently and at the same time in each subdomain.

\begin{remark}
\label{Rem:Riesz}
In the implementation of Algorithm~\ref{Alg:DD}, it is convenient to identify Euclidean spaces with their dual spaces.
Then, Riesz isomorphisms $J_{\tV}$ and $J_{\tV_s}$ become the identity operators.
\end{remark}

\begin{remark}
\label{Rem:local}
Since the term $\frac{\eta}{2}  \| \tu_s - \hu_s^{(n+1)} \|_{2, \tV_s}^2$ in~\eqref{local2} is $\eta$-strongly convex, faster algorithms which utilize the strong convexity of the energy functional can be used.
Similar observations were made in~\cite{LNP:2019,LPP:2019}.
For instance, if the full-dimension problem~\eqref{general} can be solved by the $O(1/n)$-convergent primal-dual algorithm~\cite[Algorithm~1]{CP:2011}, then we can apply the $O(1/n^2)$-convergent one~\cite[Algorithm~2]{CP:2011} to~\eqref{local2} with little modification.
Details will be given in Section~\ref{Sec:Applications}.
\end{remark}

\begin{remark}
\label{Rem:overlapping}
Even though we have assumed that the domain decomposition $\{ \Omega_s \}_{s=1}^{\mathcal{N}}$ is nonoverlapping, it is also possible to construct a decoupled augmented Lagrangian method corresponding to the case of general overlapping domain decomposition.
However, in that case, the computation and communication costs for $P_B$ becomes larger, which may cause a bottleneck in parallel computation.
\end{remark}

Under the assumption that $\tE (\tu)$ is convex, one can obtain several desired convergence properties for Algorithm~\ref{Alg:DD}.
We summarize the convergence theorems for Algorithm~\ref{Alg:DD} in Theorems~\ref{Thm:global} and~\ref{Thm:rate}.
Theorem~\ref{Thm:global} ensures the global convergence of the method and Theorem~\ref{Thm:rate} presents the convergence rate.
The proofs of those theorems can be obtained by similar arguments as~\cite{HY:2015,WT:2010}, and will be presented in Appendix~\ref{App:analysis} for the sake of completeness.

\begin{theorem}
\label{Thm:global}
Assume that $\tE$ is convex. 
Then the sequence $\{ (\tu^{(n)} , \lambda^{(n)}) \}$ generated by Algorithm~\ref{Alg:DD} converges to a critical point of the saddle point problem
\begin{equation}
\label{saddle_pb}
\min_{\tu \in \tV} \max_{\lambda \in \tV^*} \left\{ \tE(\tu) + \langle J_{\tV} (I-P_B )\tu, \lambda \rangle + \frac{\eta}{2} \| (I -P_B ) \tu \|_{2}^2 \right\}. 
\end{equation}
\end{theorem}

\begin{theorem}
\label{Thm:rate}
Assume that $\tE$ is convex.
Then the sequence $\{ (\tu^{(n)} , \lambda^{(n)}) \}$ generated by Algorithm~\ref{Alg:DD} satisfies
\begin{multline*}
\| P_B (\tu^{(n)} - \tu^{(n+1)}) \|_2^2 + \frac{1}{\eta^2} \| \lambda^{(n)} - \lambda^{(n+1)} \|_2^2 \\
\leq \frac{1}{n+1} \inf_{(\tu^* , \lambda^*)} \left( \| P_B (\tu^{(0)} - \tu^* )  \|_2^2 + \frac{1}{\eta^2} \| \lambda^{(0)} - \lambda^* \|_2^2 \right), \quad n \geq 0,
\end{multline*}
where the infimum is taken over all critical points of~\eqref{saddle_pb}.
\end{theorem}

\section{Applications}
\label{Sec:Applications}
In this section, we provide several applications of the proposed DDM for variational imaging problems.
All experiments in this section were implemented in ANSI C with OpenMPI, compiled by Intel Parallel Studio XE, and performed on a computer cluster consisting of seven machines, where each machine has two Intel Xeon SP-6148 CPUs~(2.4GHz, 20C), 192GB RAM, and the operating system CentOS~7.4 64bit.

Let $W$ be the collection of all functions from $\Omega$ to $\mathbb{R}^2$.
For $\p = (p^1 , p^2 ) \in W$, the pointwise absolute value $|\p| \in V$ is given by
\begin{equation*}
|\p|_{ij} = \sqrt{(p_{ij}^1)^2 + (p_{ij}^2)^2 }, \quad 1 \leq i \leq M \textrm{ and } 1 \leq j \leq N.
\end{equation*}
The $p$-norm of $\p \in W$ is computed as
\begin{equation*}
\| \p \|_{p, W} = \|  \, | \p | \, \|_{p, V}.
\end{equation*}

The standard forward/backward finite difference operators on $V$ with the homogeneous Neumann boundary condition are defined as follows:
\begin{equation*} \begin{split}
(D_x^+ u)_{ij} &= \begin{cases} u_{i+1, j} - u_{ij} & \textrm{ if } i < M, \\ 0 & \textrm{ if } i = M, \end{cases} \quad
(D_x^- u)_{ij} = \begin{cases} 0 & \textrm{ if } i = 1, \\ u_{ij} - u_{i-1, j} & \textrm{ if } i > 1, \end{cases}\\
(D_y^+ u)_{ij} &= \begin{cases} u_{i, j+1} - u_{ij} & \textrm{ if } j < N, \\ 0 & \textrm{ if } j = N, \end{cases} \quad 
(D_y^- u)_{ij} = \begin{cases} 0 & \textrm{ if } j = 1, \\ u_{ij} - u_{i , j-1} & \textrm{ if }j > 1. \end{cases}
\end{split} \end{equation*}
Then, the discrete gradient $\nabla^{\pm}$:~$V \rightarrow W$ is defined as
\begin{equation}
\label{d_grad}
\nabla^{\pm} u = (D_x^{\pm} u , D_y^{\pm} u).
\end{equation}
With the discrete operators defined above, a discrete total variation regularizer $TV(u)$ is given by
\begin{equation*}
TV(u) =  \| \nabla^+ u \|_{1, W}.
\end{equation*}
Similarly to~\eqref{V_s}, we define the subspace $W_s$ of $W$ by
\begin{equation*}
W_s = \left\{ \p \in W : \supp \p \subset \Omega_s \right\}.
\end{equation*}
for all $s = 1, \dots, \N$.

To discretize higher order models, we need to introduce the notion of tensor fields.
Let $\W$ be the second order tensor fields on $\Omega$.
For $\P = \begin{bmatrix}p^{11} & p^{12} \\ p^{21} & p^{22} \end{bmatrix} \in \W$, the pointwise absolute value $|\P| \in V$ is given by
\begin{equation*}
|\P|_{ij} = \sqrt{(p_{ij}^{11})^2 + (p_{ij}^{12})^2 + (p_{ij}^{21})^2 + (p_{ij}^{22})^2}, \quad 1 \leq i \leq M \textrm{ and } 1 \leq j \leq N.  
\end{equation*}
The $p$-norm of $\P \in \W$ is defined as
\begin{equation*}
\| \P \|_{p, \W} = \| \, |\P| \, \|_{p, V}.
\end{equation*}
The notion of discrete gradient given in~\eqref{d_grad} is easily extended as $\nabla^{\pm}$:~$W \rightarrow \W$.
In this setting, a discrete Hessian is defined as $\nabla^- \nabla^+ $:~$V \rightarrow \W$.
Also, we define the local space $\W_s$ on $\Omega_s$ similarly to~\eqref{V_s} by
\begin{equation*}
\W_s = \left\{ \P \in \W : \supp \P \subset \Omega_s \right\}
\end{equation*}
for all $s = 1, \dots, \N$.

\subsection{Convex Chan--Vese model for image segmentation}
In~\cite{CEN:2006}, a convex version of the Chan--Vese model~\cite{CV:2001} was proposed in the sense that thresholding a solution of
\begin{equation}
\label{CCV}
\min_{u \in BV(\Omega)} \left\{ \alpha \intO \left[ u ( f - c_1)^2 + (1-u) (f-c_2 )^2 \right] \, dx + \chi_{\left\{ 0 \leq \cdot \leq 1\right\} } (u) + TV(u) \right\}
\end{equation}
yields a global minimizer of the Chan--Vese model.
Here, $f$ is a given image and $c_1$, $c_2$ are predetermined intensity values.
The characteristic function $\chi_{\left\{ 0 \leq \cdot \leq 1\right\} } (u) $ is defined as
\begin{equation*}
\chi_{\left\{ 0 \leq \cdot \leq 1 \right\} } (u) = \begin{cases} 0 & \textrm{ if } 0 \leq u(x) \leq 1 \quad \forall x \in \Omega , \\ \infty & \textrm{ otherwise.} \end{cases}
\end{equation*}

\begin{figure}[]
\centering
\subfloat[][Convex Chan--Vese]{ \includegraphics[width=3.8cm]{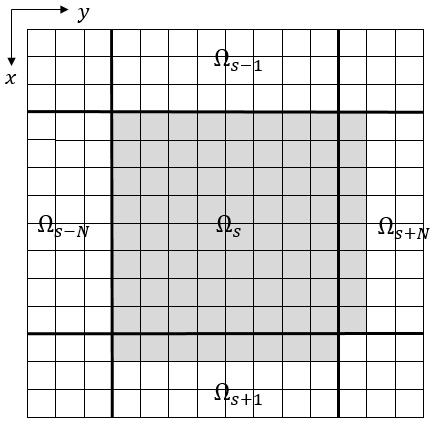} }
\subfloat[][$TV$-$L^1$ deblurring ($5\times 5$ kernel)]{ \includegraphics[width=3.8cm]{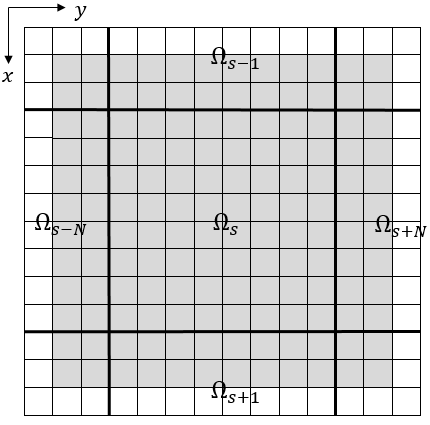} }
\subfloat[][Hessian-$L^1$ denoising]{ \includegraphics[width=3.8cm]{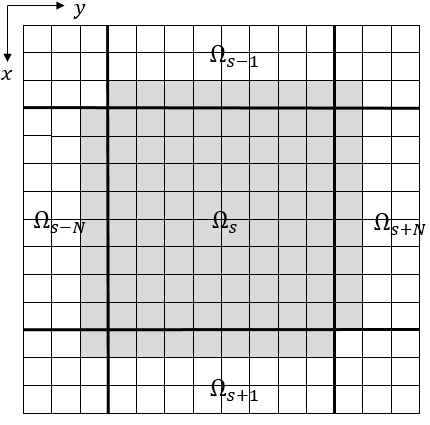} } 
\caption{Essential domains $\ed_{\Omega_s} (T)$ on the subdomain $\Omega_s$ for various examples.}
\label{Fig:ed}
\end{figure}

We may write a discretized version of~\eqref{CCV} as
\begin{equation}
\label{d_CCV}
\min_{u \in V} \left\{ E (u) = \alpha \langle u, g \rangle_V + \chi_{\left\{ 0 \leq \cdot \leq 1\right\} } (u) + \| \nabla^+ u\|_{1, W} \right\},
\end{equation}
where $g = (f - c_1)^2 - (f-c_2)^2$.
One can observe that \eqref{d_CCV} may be regarded as a total variation-regularized problem, i.e., it is expressed as the form of~\eqref{general} with 
\begin{equation*}
A = I , \quad F(u) = \alpha \langle u, g \rangle _V+ \chi_{\left\{ 0 \leq \cdot \leq 1\right\} } (u) ,  \quad \textrm{and} \quad  R(u) = \| \nabla^+ u \|_{1,W}.
\end{equation*}
In addition, $T(u)$ defined in~\eqref{integral} is given by
\begin{equation*}
T(u) = \alpha u g + \chi_{\left\{ 0 \leq \cdot \leq 1\right\}}(u) + |\nabla^+ u|.
\end{equation*}
It is clear that $F(Au)$ is separable.
That is, computation of the term $[ \alpha u g + \chi_{\left\{ 0 \leq \cdot \leq 1\right\}}(u) ] |_{\Omega_s}$ does not need the values of $u$ outside $\Omega_s$ for $s = 1, \dots, \N$.
Thus, we do not need to consider that term for the construction of essential domain, i.e.,
\begin{equation*}
\tOmega_s = \ed_{\Omega_s} (T) = \ed_{\Omega_s} (|\nabla^+|).
\end{equation*}
Since $\nabla^+$ is the forward difference,
the essential domain of $|\nabla^+ |$ on $\Omega_s$ becomes
\begin{equation*}
\ed_{\Omega_s} (| \nabla^+ |) = \bigcup_{(i,j) \in \Omega_s} \left\{ (i,j), (i+1,j), (i, j+1) \in \Omega \right\}.
\end{equation*}
See Fig.~\ref{Fig:ed}(a) for the graphical description of $\ed_{\Omega_s} (| \nabla^+ |)$.

With $\{ \tOmega_s \}$ defined above, one can easily check that the formulation~\eqref{DD} is exactly the same as the one proposed in~\cite{DCT:2016}.
If we apply the first order primal-dual algorithm~\cite{CP:2011} to~\eqref{DD}, then we obtain~\cite[Algorithm~III]{DCT:2016}.
That is,~\cite[Algorithm~III]{DCT:2016} reads as
\begin{eqnarray*}
p^{(n+1)} &=& p^{(n)} + \sigma B(2\tu^{(n+1)} - \tu^{(n)}), \\
\tu^{(n+1)} &=& \argmin_{\tu \in \tV} \left\{ \tE (\tu) + \frac{1}{2\tau} \| \tu - (\tu^{(n)} - \tau B^* p^{(n+1)}) \|_{2, \tV}^2 \right\}
\end{eqnarray*}
for some $\sigma$, $\tau > 0$.
In this sense, we can say that~\cite[Algorithm~III]{DCT:2016} and the proposed DDM solve the same problem~\eqref{DD} but employs different solvers.
While the variable $p^{(n)}$ of~\cite{DCT:2016} lies on the subdomain interfaces, the variable $\lambda^{(n)}$ of the proposed DDM is distributed in each subdomain.
Consequently, the computation of $\lambda^{(n)}$ in the proposed DDM has an advantage in view of parallel computation compared to $p^{(n)}$ in~\cite{DCT:2016}.

Even though the authors of~\cite{DCT:2016} claimed that their proposed DDM is a nonoverlapping one, we regard it as an overlapping one since it is more natural to consider a line of pixels in the image as a subset of positive measure.
This issue will be discussed further in Appendix~\ref{App:DCT}.

Next, we consider how to treat local problems.
By~\eqref{local2}, the general form of local problems in $\tOmega_s$ for~\eqref{d_CCV} is given by
\begin{equation}
\label{local_seg}
\min_{\tu_s \in \tV_s} \left\{ \alpha \langle \tu_s|_{\Omega_s}, g|_{\Omega_s} \rangle_{V_s} + \chi_{\left\{ 0 \leq \cdot \leq 1 \right\}} (\tu_s |_{\Omega_s}) + \| \nabla^+|_{\Omega_s} \tu_s \|_{W_s} + \frac{\eta}{2} \| \tu_s - \hu_s \|_{2, \tV_s}^2 \right\}
\end{equation}
for some $\hu_s \in \tV_s$.
As we noticed in Remark~\ref{Rem:local},~\eqref{local_seg} can be efficiently solved by the $O(1/n^2)$-convergent primal-dual algorithm~\cite{CP:2011} applied to its primal-dual form
\begin{equation*} \begin{split}
\label{local_CV_pd_temp}
\min_{\tu_s \in \tV_s} \max_{\p_s \in W_s} \bigg\{ &\langle \nabla^+ |_{\Omega_s} \tu_s  , \p_s \rangle_{W_s} + \alpha \langle \tu_s |_{\Omega_s} , g|_{\Omega_s} \rangle_{V_s} + \chi_{\left\{ 0 \leq \cdot \leq 1 \right\}} (\tu_s |_{\Omega_s}) \\
&+ \frac{\eta}{2} \| \tu_s - \hu_s \|_{2, \tV_s}^2 - \chi_{\left\{ |\cdot| \leq 1\right\}} (\p_s ) \bigg\}.
\end{split} \end{equation*}
We summarize the primal-dual algorithm for~\eqref{local_seg} in Algorithm~\ref{Alg:local_seg}.

\begin{algorithm}[]
\caption{Primal-dual algorithm for the local segmentation problem \eqref{local_seg}}
\begin{algorithmic}[]
\label{Alg:local_seg}
\STATE Choose $\sigma_0, \tau_0 > 0$ with $\sigma_0 \tau_0 \leq 1/8$ and $0 \leq \gamma \leq \eta$.
Let $\tu_s^{(0)} \in \tV_s$ and $\p_s^{(0)} \in W_s$.
\FOR{$n=0,1,2, \dots$}
\STATE $\displaystyle \p_s^{(n+1)} = \proj_{\left\{ | \cdot | \leq 1 \right\}} \left( \p_s^{(n)} + \sigma_n \nabla^+ |_{\Omega_s} \bu_s^{(n)} \right)$
\STATE $\displaystyle \tu_s^{(n+1)} = \proj_{\left\{ 0 \leq \cdot \leq 1 \right\} } \left( \frac{\tu_s^{(n)} - \tau_n \left[ (\nabla^+ |_{\Omega_s})^* \p_s  + \alpha g |_{\Omega_s} \right] + \tau_n \eta \hu_s}{1 + \tau_n \eta} \right)$
\STATE $\theta_n = 1/\sqrt{1 + 2\gamma \tau_n}$, $\tau_{n+1} = \theta_n \tau_n$, $\sigma_{n+1} = \sigma_n / \theta_n$
\STATE $\bu_s^{(n+1)} = (1+\theta_n )\tu_s^{(n+1)} - \theta_n \tu_s^{(n)}$
\ENDFOR
\end{algorithmic}
\end{algorithm}

The condition $\sigma_0 \tau_0 \leq 1/8$ is derived from the fact that the operator norm of $\nabla^+$ has a bound $\| \nabla^+ \|^2 \leq 8$~\cite[Theorem~3.1]{Chambolle:2004}.
Two projection operators appearing in Algorithm~\ref{Alg:local_seg} are easily computed by pointwise Euclidean projections.

\begin{figure}[]
\centering
\subfloat[][Original image]{ \includegraphics[height=3.8cm]{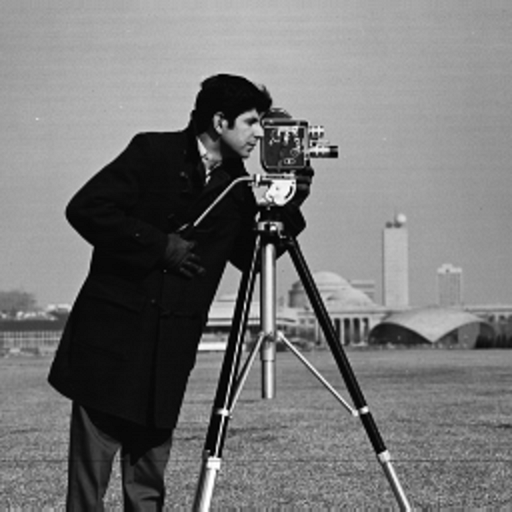} }
\subfloat[][$\N=1$]{ \includegraphics[height=3.8cm]{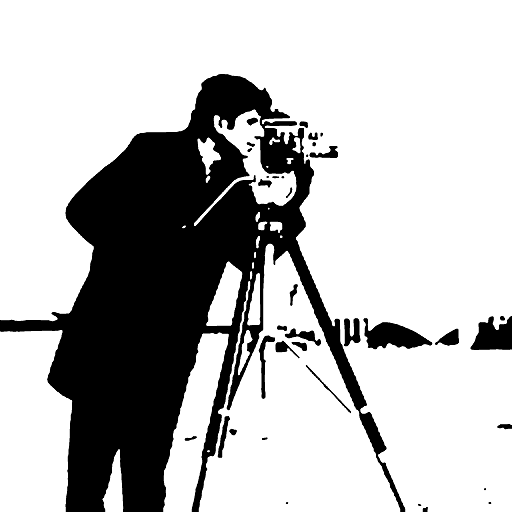} }
\subfloat[][$\N=16 \times 16$]{ \includegraphics[height=3.8cm]{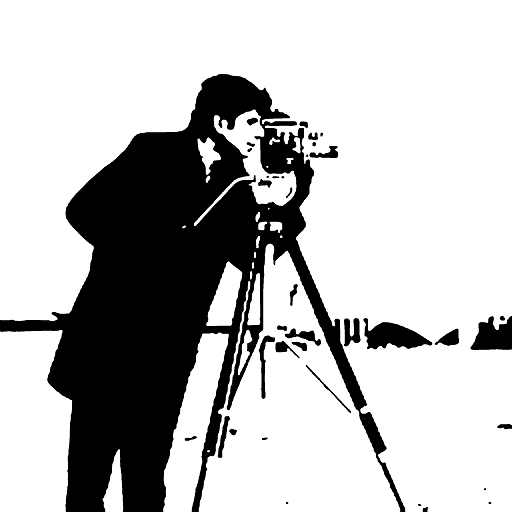} }
\caption{Image results for the image segmentation problem~\eqref{d_CCV}.}
\label{Fig:seg}
\end{figure}
\begin{figure}[]
\centering
\subfloat[][]{ \includegraphics[height=4.0cm]{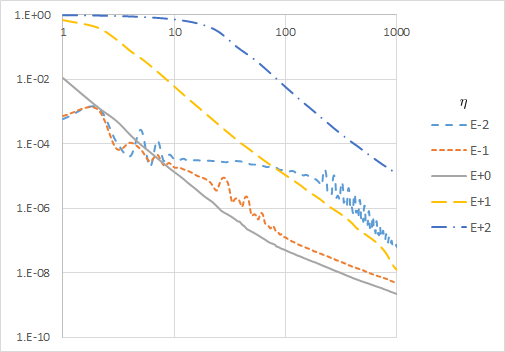} }
\subfloat[][]{ \includegraphics[height=4.0cm]{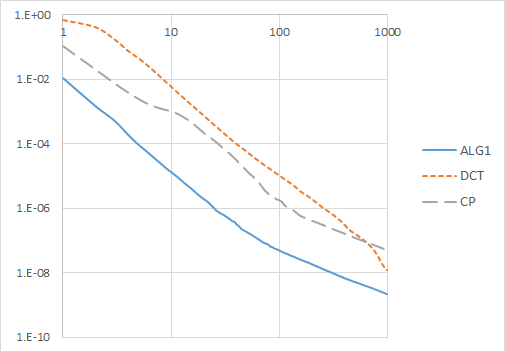} }
\caption{Decay of $\frac{E(u^{(n)}) - E^*}{|E^*|}$ of Algorithm~\ref{Alg:DD} for the image segmentation problem~\eqref{d_CCV} ($\N = 8 \times 8$): \textbf{(a)} various $\eta$, \textbf{(b)} comparison with other methods.}
\label{Fig:energy_seg}
\end{figure}

\begin{table} \centering
\begin{tabular}{| c | c c |} \hline
$\N$ & \#iter& time (sec) \\
\hline
$1$ & - &  31.40 \\
$2\times2$ &73 & 45.30 \\
$4\times4$ & 73 &  15.21 \\
$8\times8$ & 72 &  6.95 \\
$16\times16$ & 72 &  4.82 \\
\hline
\end{tabular}
\captionof{table}{Performance of Algorithm~\ref{Alg:DD} for the image segmentation problem~\eqref{d_CCV}, $\eta = 1$.}
\label{Table:seg}
\end{table}

In order to assess the efficiency of the proposed method, we present several numerical results for Algorithm~\ref{Alg:DD} applied to~\eqref{d_CCV}.
A test image ``Cameraman $2048 \times 2048$'' was used for our experiments~(see Fig.~\ref{Fig:seg}(a)). 
Parameters in~\eqref{d_CCV} were set as $\alpha = 10$, $c_1 = 0.6$, $c_2 = 0.1$, and intial guesses for Algorithm~\ref{Alg:DD} as $\tu^{(0)} = 0$, $\lambda^{(0)} = 0$. 
Since local problems need not to be solved exactly~\cite{CTWY:2015,DCT:2016}~(see also Remark~\ref{Rem:inexact}), local problems~\eqref{local_seg} were approximately solved by 10 iterations of Algorithm~\ref{Alg:local_seg} with $\sigma_0 = \tau_0 = 1/\sqrt{8}$ and $\gamma = 0.125\eta$.
The number of iterations was chosen heuristically to reduce the wall-clock time.

Fig.~\ref{Fig:energy_seg}(a) shows the decay of $\frac{E(u^{(n)}) - E(u^*)}{|E(u^*)|}$ of Algorithm~\ref{Alg:DD} with various penalty parameters $\eta$, where $\N = 8 \times 8$, $u^{(n)} = P_B \tu^{(n)}$, and $E^*$ is the minimum energy computed by $10^6$ iterations of the primal-dual algorithm.
We can observe that the decay of the energy highly depends on $\eta$.
For small $\eta$, we have small $E(u^{(1)})$.
However, $\eta$ should be sufficiently large to accomplish fast convergence rate.

We compared the proposed method with other existing methods for~\eqref{d_CCV} in Fig.~\ref{Fig:energy_seg}(b). 
The following algorithms were used in our experiments:
\begin{itemize}
\item ALG1: Algorithm~\ref{Alg:DD}, $\N = 8 \times 8$, $\eta = 1$.
\item DCT: DDM proposed by Duan, Chang, and Tai~\cite{DCT:2016}, $\N = 8 \times 8$, $\tau = 1$, $\sigma \tau = 1/8$.
\item CP: Primal-dual algorithm proposed by Chambolle and Pock~\cite{CP:2011}, $\sigma = \tau = 1/\sqrt{8}$.
\end{itemize}
We note that we used the $O(1/n^2)$-convergent primal-dual algorithm~\cite[Algorithm~2]{CP:2011} for local problems of DCT instead of~\cite[Algorithm~II]{DCT:2016}.
ALG1 outperforms both DCT and CP in the sense of the energy decay.
On the other hand, ALG1 has an advantage compared to DCT that its parallel implementation is easy because the Lagrange multiplier can be distributed in each processor.

To highlight the efficiency of the proposed method as a parallel solver, we present the timing results with various numbers of subdomains $\N$ in Table~\ref{Table:seg}.
We used the stop criterion
\begin{equation}
\label{stop}
\max \left\{ \left| \frac{E(u^{(n)}) - E(u^{(n+1)})}{E(f)} \right| , \frac{\| u^{(n)} - u^{(n+1)} \|_2 }{\| f\|_2}\right\} < \TOL,
\end{equation}
where $\TOL = 10^{-4}$.
The full-dimension problem~($\N = 1$) was solved by the primal-dual algorithm with the same parameter setting as local problems.
The number of iterations of Algorithm~\ref{Alg:DD} is abbreviated as \#iter.
The wall-clock time reduces as $\N$ grows.
In particular,~\eqref{d_CCV} can be solved in few seconds if we use sufficiently many subdomains.
Figs.~\ref{Fig:seg}(b) and~(c) show the results in the cases $\N = 2 \times 2$ and $\N = 16 \times 16$ thresholded by $1/2$, respectively.
Since two results are not visually distinguishable, it is ensured that the proposed method gives a reliable solution even if $\N$ is large.

\subsection{$TV$-$L^1$ model for image deblurring}
In the $TV$-$L^1$ model~\cite{CE:2005,Nikolova:2004}, the fidelity term and the regularizer are given by the $L^1$-norm and the total variation, respectively.
The discrete $TV$-$L^1$ model for image deblurring is stated as
\begin{equation}
\label{TVL1}
\min_{u \in V} \left\{ E(u) := \alpha \| Au - f \|_{1, V} + \| \nabla^+ u \|_{1,W} \right\},
\end{equation}
where $A$:~$V \rightarrow V$ is a blur kernel.
In this case, $T(u)$ defined in~\eqref{integral} is given by
\begin{equation*}
T(u) = |Au - f| + |\nabla^+ u|,
\end{equation*}
where $f \in V$ is a corrupted image.
Clearly, we have
\begin{equation*}
\ed_{\Omega_s}(T) = \ed_{\Omega_s}(A) \cup \ed_{\Omega_s}(|\nabla^+|).
\end{equation*}
If the blur kernel has the size $(2l+1) \times (2l+1)$ for $l \in \mathbb{Z}_{>0}$,
computation of $Au$ at a point requires the values of $u$ in the $(2l+1) \times (2l+1)$ square centered at the point.
Thus, the essential domain of $A$ on $\Omega_s$ consists of $\Omega_s$ itself and the band of width $l$ enclosing $\Omega_s$.
More precisely, it is expressed as
\begin{equation*}
\ed_{\Omega_s} (A) = \bigcup_{(i,j) \in \Omega_s} \bigcup_{-l \leq a,b \leq l} \left\{ (i+a, j+b) \in \Omega \right\}.
\end{equation*}
Since $\ed_{\Omega_s} (|\nabla^+ |) \subset \ed_{\Omega_s} (A)$, we have
\begin{equation*}
\tOmega_s = \ed_{\Omega_s} (T) = \ed_{\Omega_s} (A).
\end{equation*}
Fig.~\ref{Fig:ed}(b) shows $\ed_{\Omega_s} (T)$ when $l = 2$, that is, a $5 \times 5$ kernel is used.

In view of~\eqref{local2}, local problems in $\tOmega_s$ for~\eqref{TVL1} have the form
\begin{equation}
\label{local_deb}
\min_{\tu_s \in \tV_s} \left\{ \alpha \| A|_{\Omega_s} \tu_s  - f|_{\Omega_s} \|_{1,V_s} + \| \nabla^+|_{\Omega_s} \tu_s \|_{1,W_s} + \frac{\eta}{2} \| \tu_s - \hu_s \|_{2,\tV_s}^2 \right\}
\end{equation}
for some $\hu_s \in \tV_s$.
Algorithm~\ref{Alg:local_deb} presents the $O(1/n^2)$-convergent primal-dual algorithm applied to a primal-dual form
\begin{equation*} \begin{split}
\min_{\tu_s \in \tV_s} \max_{\p_s \in W_s, q_s \in V_s} \bigg\{ &\langle \nabla|_{\Omega_s} \tu_s, \p_s \rangle_{W_s} + \langle A|_{\Omega_s} \tu_s - f|_{\Omega_s} , q_s \rangle_{V_s} + \frac{\eta}{2} \| \tu_s - \hu_s \|_{2, \tV_s}^2 \\
&- \chi_{\left\{ | \cdot | \leq 1\right\} } (\p_s) - \chi_{\left\{ |\cdot | \leq \alpha \right\}} (q_s) \bigg\}
\end{split} \end{equation*}
of~\eqref{local_deb}.
For details on the derivation of the above primal-dual form, see Section~2 of~\cite{LNP:2019}.

\begin{algorithm}[]
\caption{Primal-dual algorithm for the local deblurring problem \eqref{local_deb}}
\begin{algorithmic}[]
\label{Alg:local_deb}
\STATE Choose $\sigma_0, \tau_0 > 0$ with $\sigma_0 \tau_0 \leq 1/9$ and $0 \leq \gamma \leq \eta$.
Let $\tu_s^{(0)} \in \tV_s$, $\p_s^{(0)} \in W_s$, and $q_s^{(0)} \in V_s$.
\FOR{$n=0,1,2, \dots$}
\STATE $\displaystyle \p_s^{(n+1)} = \proj_{\left\{ | \cdot | \leq 1 \right\}} \left( \p_s^{(n)} + \sigma_n \nabla^+ |_{\Omega_s} \bu_s^{(n)} \right)$
\STATE $\displaystyle q_s^{(n+1)} = \proj_{\left\{ |\cdot| \leq \alpha \right\}} \left( q_s^{(n)} + \sigma_n (A|_{\Omega_s} \bu_s^{(n)} - f|_{\Omega_s}) \right)$
\STATE $\displaystyle \tu_s^{(n+1)} = \frac{\tu_s^{(n)} - \tau_n \left[ (\nabla^+|_{\Omega_s})^* \p_s  + (A |_{\Omega_s})^* q_s\right] + \tau_n \eta \hu_s }{1 + \tau_n \eta} $
\STATE $\theta_n = 1/\sqrt{1 + 2\gamma \tau_n}$, $\tau_{n+1} = \theta_n \tau_n$, $\sigma_{n+1} = \sigma_n / \theta_n$
\STATE $\bu_s^{(n+1)} = (1+\theta_n )\tu_s^{(n+1)} - \theta_n \tu_s^{(n)}$
\ENDFOR
\end{algorithmic}
\end{algorithm}

The condition~$\sigma_0 \tau_0 \leq 1/9$ in Algorithm~\ref{Alg:local_deb} is due to that $\| \nabla^+ \|^2 \leq 8$ and $\| A \|^2 \leq 1$~\cite[Proposition~4]{LNP:2019}.
Similarly to Algorithm~\ref{Alg:local_seg}, projection operators in Algorithm~\ref{Alg:local_deb} are accomplished by pointwise Euclidean projections.

\begin{figure}[]
\centering
\subfloat[][Corrupted (PSNR: 29.16)]{ \includegraphics[height=3.8cm]{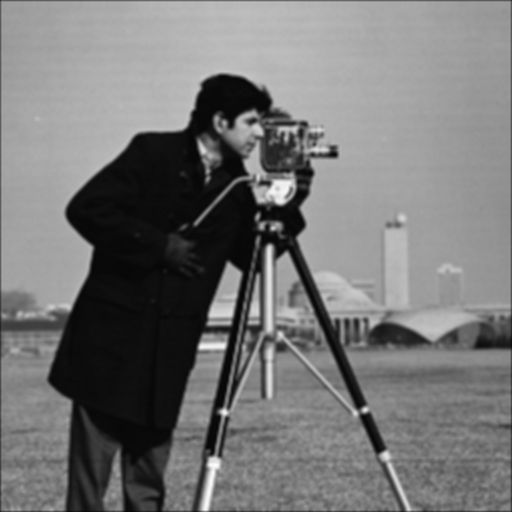} }
\subfloat[][$\N=1$ (PSNR: 40.48)]{ \includegraphics[height=3.8cm]{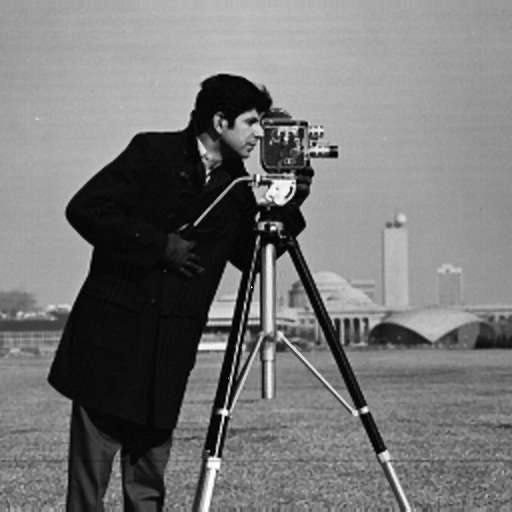} }
\subfloat[][$\N=16 \times 16$ (PSNR: 41.47)]{ \includegraphics[height=3.8cm]{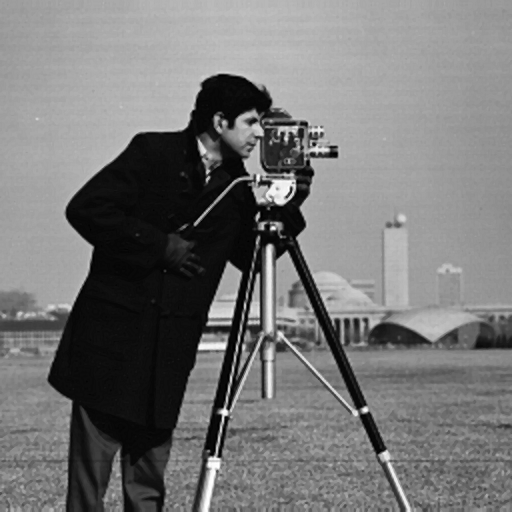} } \\
\subfloat[][Corrupted (PSNR: 24.04)]{ \includegraphics[height=3.8cm]{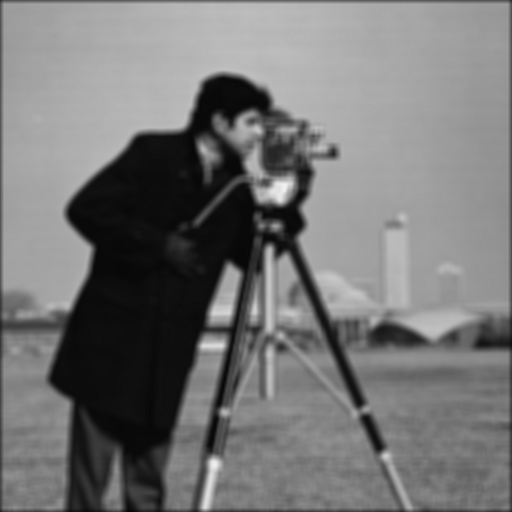} }
\subfloat[][$\N=1$ (PSNR: 35.16)]{ \includegraphics[height=3.8cm]{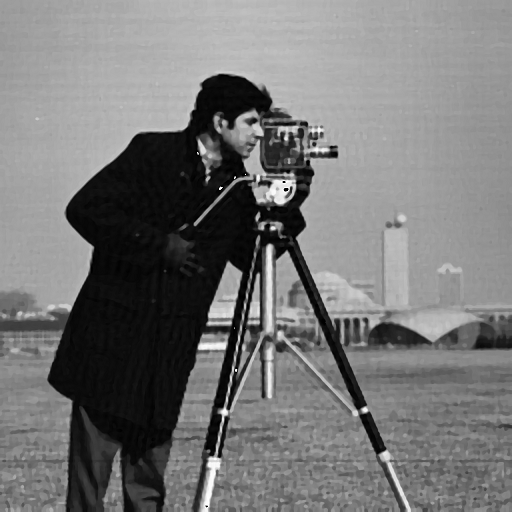} }
\subfloat[][$\N=16 \times 16$ (PSNR: 35.91)]{ \includegraphics[height=3.8cm]{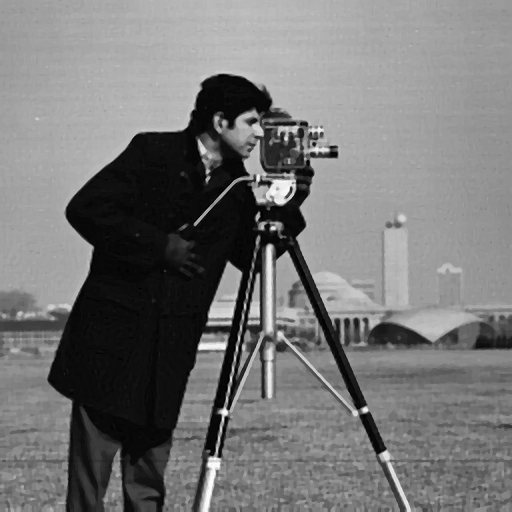} }
\caption{Image results for the image deblurring problem~\eqref{TVL1}. \textbf{(a--c)} $17 \times 17$ kernel, \textbf{(d--f)} $33 \times 33$ kernel.}
\label{Fig:deb}
\end{figure}
\begin{figure}[]
\centering
\subfloat[][$17 \times 17$ kernel]{ \includegraphics[height=4.0cm]{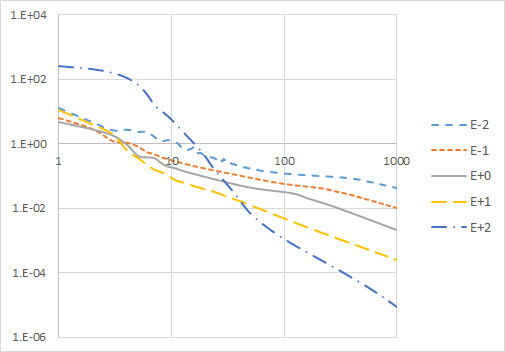} }
\subfloat[][$17 \times 17$ kernel]{ \includegraphics[height=4.0cm]{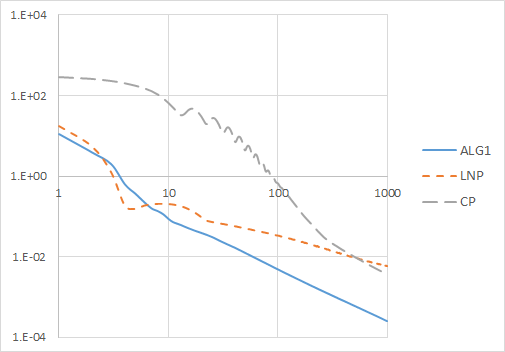} } \\
\subfloat[][$33 \times 33$ kernel]{ \includegraphics[height=4.0cm]{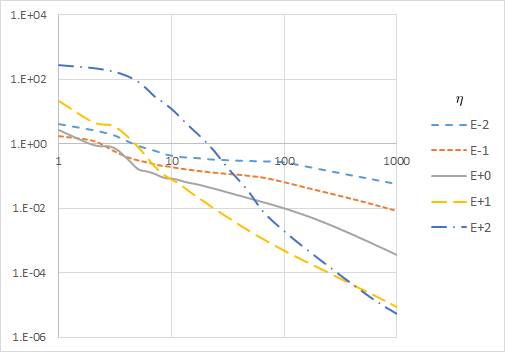} }
\subfloat[][$33 \times 33$ kernel]{ \includegraphics[height=4.0cm]{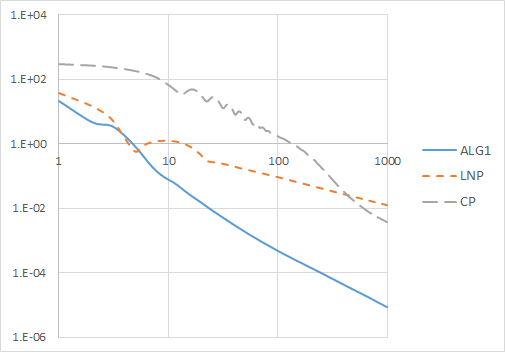} } \\
\caption{Decay of $\frac{E(u^{(n)}) - E^*}{|E^*|}$ of Algorithm~\ref{Alg:DD} for the image deblurring problem~\eqref{TVL1} ($\N = 8 \times 8$): \textbf{(a, c)} various $\eta$, \textbf{(b, d)} comparison with other methods.}
\label{Fig:energy_deb}
\end{figure}

\begin{table} \centering
\begin{tabular}{| c | c c c | c c c |} \hline
\multirow{2}{*}{$\N$} & \multicolumn{3}{c|}{$17 \times 17$ kernel} & \multicolumn{3}{c|}{$33 \times 33$ kernel} \\
\cline{2-7}
& \#iter& time (sec) & PSNR & \#iter& time (sec) & PSNR \\
\hline
$1$ & - &  1240.02 & 40.48 & - &  6890.73 & 35.16\\
$2\times2$ & 16 & 1878.74 & 41.71 & 16 & 11286.46 & 36.09 \\
$4\times4$ & 21 &  464.33 & 41.67 & 21 &  2919.95 & 36.07 \\
$8\times8$ & 25 & 136.33 & 41.62 & 25 & 838.28 & 36.02 \\
$16\times16$ & 31 & 39.22 & 41.47 & 31 & 231.90 & 35.91  \\
\hline
\end{tabular}
\captionof{table}{Performance of Algorithm~\ref{Alg:DD} for the image deblurring problem~\eqref{TVL1}, $\eta = 10$.}
\label{Table:deb}
\end{table}

Now, we provide numerical results for Algorithm~\ref{Alg:DD} for~\eqref{TVL1}.
In~\eqref{TVL1}, the images $f$ were made by applying the $17 \times 17$ and $33 \times 33$ average kernels to the ``Cameraman $2048 \times 2048$'' test image; see Figs.~\ref{Fig:deb}(a) and~(d), respectively.
We set $\alpha = 10$, $\tu^{(0)} = 0$, and $\lambda^{(0)} = 0$.
Local problems~\eqref{local_deb} were solved by 50 iterations of Algorithm~\ref{Alg:local_deb} with $\sigma_0 = \tau_0 = 1/3$ and $\gamma = 0.125 \eta$.
The number of iterations was optimized heuristically with respect to the wall-clock time.

The energy decay of Algorithm~\ref{Alg:DD} for~\eqref{TVL1} with various penalty parameters $\eta$ is shown in Figs.~\ref{Fig:energy_deb}(a) and~(c), where $\N = 8 \times 8$ and $u^{(n)} = P_B \tu^{(n)}$.
The minimum energy $E^*$ is computed by $10^6$ iterations of the primal-dual algorithm.
The behavior of the energy decay with respect to $\eta$ is similar to the case of~\eqref{d_CCV}.

Figs.~\ref{Fig:energy_deb}(b) and~(d) provide the comparison of the energy decay with existing methods for~\eqref{TVL1}:
\begin{itemize}
\item ALG1: Algorithm~\ref{Alg:DD}, $\N = 8 \times 8$, $\eta = 10$.
\item LNP: DDM proposed by Lee, Nam, and Park~\cite{LNP:2019}, $\N = 8\times 8$, $\tau = 0.1$, $\sigma \tau = 1/9$.
\item CP: Primal-dual algorithm proposed by Chambolle and Pock~\cite{CP:2011}, $\tau = 0.02$, $\sigma\tau = 1/9$.
\end{itemize}
Similarly to the case of~\eqref{d_CCV}, ALG1 outperforms other methods in the sense of the energy decay.
We also note that ALG1 does not have data structure on the subdomain interfaces whereas LNP has.

Table~\ref{Table:deb} shows the wall-clock time of the proposed method for various number of subdomains $\N$.
The stop criterion for the experiments in Table~\ref{Table:deb} is~\eqref{stop} with $\mathrm{TOL} = 10^{-3}$.
The case $\N = 1$ was solved by the primal-dual algorithm with the parameter setting described above.
The wall-clock time decreases as $\N$ increases.
PSNRs~(peak signal-to-noise ratios) are slightly different for $\N$, but all are large enough.
We note that such difference arises due to the nonuniqueness of a solution of~\eqref{TVL1}.
The results for $\N=2 \times 2$ and $\N= 16 \times 16$ displayed in Figs.~\ref{Fig:deb}(b),~(e) and~(c),~(f), respectively, are not visually distinguishable.
In particular, they show no trace on the subdomain interfaces.

\subsection{Hessian-$L^1$ denoising}
We noted before that the proposed DDM can be applied to not only first order models but also higher order models.
As a higher order model problem, we consider the following discrete Hessian-regularized problem:
\begin{equation}
\label{LLT}
\min_{u \in V} \left\{ E(u) := \alpha \| A u - f \|_{1, V} + \| \nabla^- \nabla^+ u \|_{1, \W} \right\},
\end{equation}
where $f \in V$ is a corrupted image.
It is readily observed that
\begin{equation*}
T(u) = | Au - f| + |\nabla^- \nabla^+ u|.
\end{equation*}
We note that the Hessian regularizer was proposed in~\cite{LLT:2003} to overcome the staircase effect of first order models.
For simplicity, we consider the denoising problem, i.e., $A = I$.
In this case, we clearly have
\begin{equation*}
\ed_{\Omega_s} (T) = \ed_{\Omega_s} (|\nabla^- \nabla^+ |).
\end{equation*}
Since $\nabla^- \nabla^+ u$ is composed of backward difference of $\nabla^+ u$, $\ed_{\Omega_s} (|\nabla^- \nabla^+ |)$ is expressed as
\begin{equation*} \begin{split}
\ed_{\Omega_s} (|\nabla^- \nabla^+ |) &= \ed_{\ed_{\Omega_s} (| \nabla^- |)} (| \nabla^+ |) \\
&= \bigcup_{(i,j) \in \ed_{\Omega_s} (|\nabla^-|)} \left\{ (i, j), (i+1, j), (i, j+1) \in \Omega \right\}.
\end{split} \end{equation*}
See Fig.~\ref{Fig:ed}(c) for a graphical description.

Similarly to~\eqref{local_deb}, the general form of local problems for~\eqref{LLT} is expressed as
\begin{equation}
\label{local_LLT}
\min_{\tu_s \in \tV_s} \left\{ \alpha \| A|_{\Omega_s} \tu_s  - f|_{\Omega_s} \|_{1,V_s} + \| (\nabla^- \nabla^+)|_{\Omega_s} \tu_s \|_{1,W_s} + \frac{\eta}{2} \| \tu_s - \hu_s \|_{2,\tV_s}^2 \right\}
\end{equation}
for some $\hu_s \in \tV_s$.
Noting that
\begin{equation*}
\| \nabla^- \nabla^+ \|^2 \leq \| \nabla^- \|^2 \| \nabla^+ \|^2 \leq 8 \cdot 8 = 64 ,
\end{equation*}
We obtain Algorithm~\ref{Alg:local_LLT} which efficiently solves~\eqref{local_LLT} in the same manner as Algorithm~\ref{Alg:local_deb}.

\begin{algorithm}[]
\caption{Primal-dual algorithm for the local problem \eqref{local_LLT}}
\begin{algorithmic}[]
\label{Alg:local_LLT}
\STATE Choose $\sigma_0, \tau_0 > 0$ with $\sigma_0 \tau_0 \leq 1/65$ and $0 \leq \gamma \leq \eta$.
Let $\tu_s^{(0)} \in \tV_s$, $\P_s^{(0)} \in \W_s$, and $q_s^{(0)} \in V_s$.
\FOR{$n=0,1,2, \dots$}
\STATE $\displaystyle \P_s^{(n+1)} = \proj_{\left\{ | \cdot | \leq 1 \right\}} \left( \P_s^{(n)} + \sigma_n (\nabla^- \nabla^+ ) |_{\Omega_s} \bu_s^{(n)} \right)$
\STATE $\displaystyle q_s^{(n+1)} = \proj_{\left\{ |\cdot| \leq \alpha \right\}} \left( q_s^{(n)} + \sigma_n (A|_{\Omega_s} \bu_s^{(n)} - f|_{\Omega_s}) \right)$
\STATE $\displaystyle \tu_s^{(n+1)} = \frac{\tu_s^{(n)} - \tau_n \left[ ((\nabla^- \nabla^+)|_{\Omega_s})^* \p_s  + (A |_{\Omega_s})^* q_s\right] + \tau_n \eta \hu_s }{1 + \tau_n \eta} $
\STATE $\theta_n = 1/\sqrt{1 + 2\gamma \tau_n}$, $\tau_{n+1} = \theta_n \tau_n$, $\sigma_{n+1} = \sigma_n / \theta_n$
\STATE $\bu_s^{(n+1)} = (1+\theta_n )\tu_s^{(n+1)} - \theta_n \tu_s^{(n)}$
\ENDFOR
\end{algorithmic}
\end{algorithm}

\begin{figure}[]
\centering
\subfloat[][Corrupted (PSNR: 12.07)]{ \includegraphics[height=3.8cm]{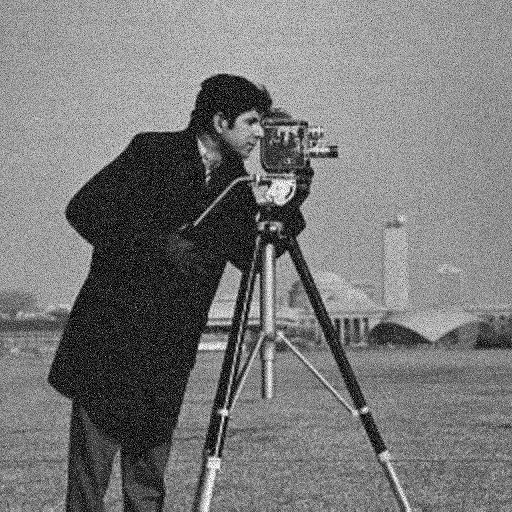} }
\subfloat[][$\N=1$ (PSNR: 57.68)]{ \includegraphics[height=3.8cm]{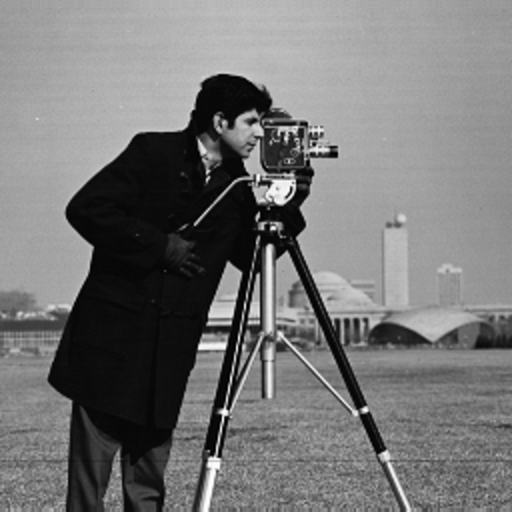} }
\subfloat[][$\N=16 \times 16$ (PSNR: 58.30)]{ \includegraphics[height=3.8cm]{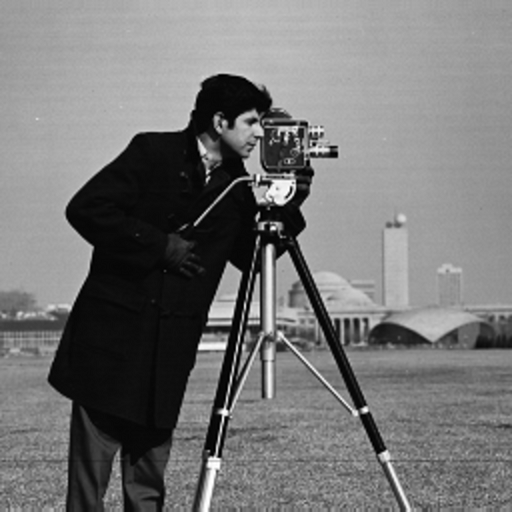} } \\
\subfloat[][Corrupted (PSNR: 9.67)]{ \includegraphics[height=3.8cm]{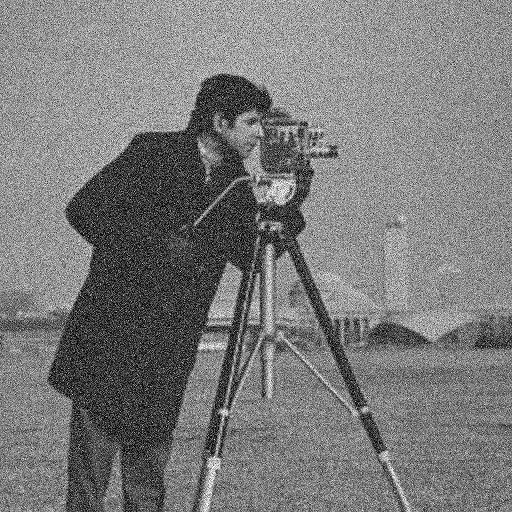} }
\subfloat[][$\N=1$ (PSNR: 54.86)]{ \includegraphics[height=3.8cm]{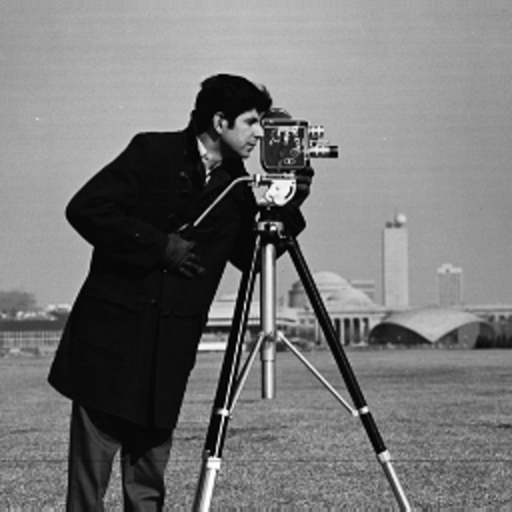} }
\subfloat[][$\N=16 \times 16$ (PSNR: 55.61)]{ \includegraphics[height=3.8cm]{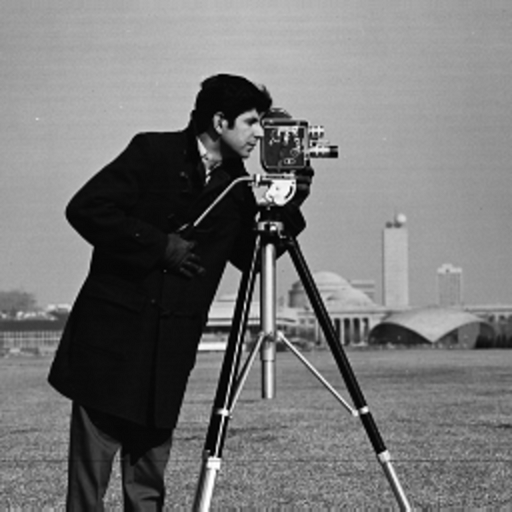} }
\caption{Image results for the Hessian-$L^1$ problem~\eqref{LLT}. \textbf{(a--c)} 20\% noise, \textbf{(d--f)} 40\% noise.}
\label{Fig:LLT}
\end{figure}

\begin{figure}[]
\centering
\subfloat[][20\% noise]{ \includegraphics[height=4.0cm]{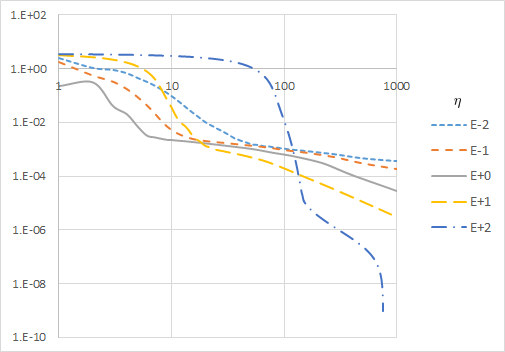} }
\subfloat[][20\% noise]{ \includegraphics[height=4.0cm]{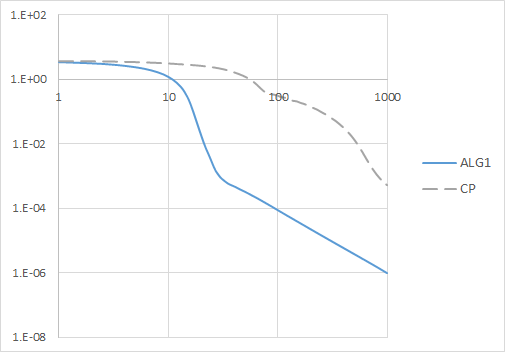} } \\
\subfloat[][40\% noise]{ \includegraphics[height=4.0cm]{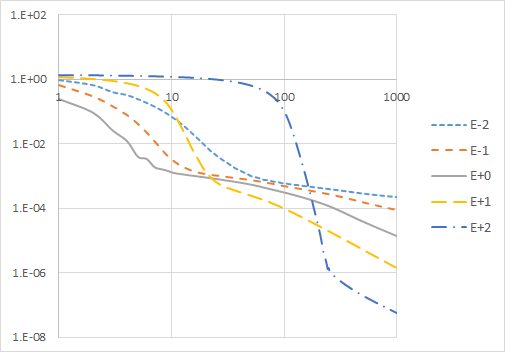} }
\subfloat[][40\% noise]{ \includegraphics[height=4.0cm]{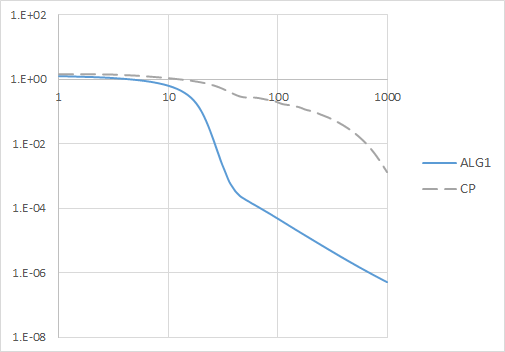} }
\caption{Decay of $\frac{E(u^{(n)}) - E^*}{|E^*|}$ of Algorithm~\ref{Alg:DD} for the Hessian-$L^1$ problem~\eqref{LLT} ($\N = 8 \times 8$): \textbf{(a, c)} various $\eta$, \textbf{(b, d)} comparison with other methods.}
\label{Fig:energy_LLT}
\end{figure}

\begin{table} \centering
\begin{tabular}{| c | c c c | c c c |} \hline
\multirow{2}{*}{$\N$} & \multicolumn{3}{c|}{20\% noise)} & \multicolumn{3}{c|}{40\% noise} \\
\cline{2-7}
& \#iter& time (sec) & PSNR & \#iter& time (sec) & PSNR \\
\hline
$1$ & - &  830.96 & 57.68 & - &  1103.97 & 54.86\\
$2\times2$ & 37 & 239.21 & 58.36 & 47 & 295.99 & 55.65 \\
$4\times4$ & 38 &  62.19 & 58.35 & 49 &  78.18 & 55.64 \\
$8\times8$ & 39 &  17.24 & 58.33 & 52 &  22.67 & 55.63 \\
$16\times16$ & 41 &  4.49 & 57.30 & 57 &  7.16 & 55.61  \\
\hline
\end{tabular}
\captionof{table}{Performance of Algorithm~\ref{Alg:DD} for the Hessian-$L^1$ problem~\eqref{LLT}, $\eta = 20$.}
\label{Table:LLT}
\end{table}

For numerical experiments of the proposed DDM for~\eqref{LLT}, we use the test images ``Cameraman $2048 \times 2048$'' corrupted by 20\% and 40\% salt-and-pepper noise; see Figs.~\ref{Fig:LLT}(a) and~(d), respectively.
We set $\alpha = 1$, $\tu^{(0)} = 0$, and $\lambda^{(0)} = 0$.
50 iterations of Algorithm~\ref{Alg:local_LLT} with $\sigma_0 = \tau_0 = 1/\sqrt{65}$ and $\gamma = 0.125 \eta$ were used as local solvers.

Figs.~\ref{Fig:energy_LLT}(a) and~(c) displays the decay $\frac{E(u^{(n)}) - E^*}{|E^*|}$ of the proposed method with $\N = 8 \times 8$ and various penalty parameters $\eta$, where $u^{(n)} = P_B \tu^{(n)}$ and $E^*$ is obtained by $10^6$ primal-dual iterations again.
The energy decay in the case~\eqref{LLT} shows a similar behavior to other methods~\eqref{d_CCV} and~\eqref{TVL1}.

In Figs.~\ref{Fig:energy_LLT}(b) and~(d), the energy decay of the following two methods for~\eqref{LLT} are plotted:
\begin{itemize}
\item ALG1: Algorithm~\ref{Alg:DD}, $\N = 8 \times 8$, $\eta = 20$.
\item CP: Primal-dual algorithm proposed by Chambolle and Pock~\cite{CP:2011}, $\tau = 0.02$, $\sigma \tau = 1/65$.
\end{itemize}
To the best of our knowledge, there is no existing DDM which accommodate higher order imaging problems such as~\eqref{LLT}.
Thus, Figs.~\ref{Fig:energy_LLT}(b) and~(d) do not contain a comparison with existing DDMs.
It is clear that ALG1 converges to the minimum faster than CP.

Table~\ref{Table:LLT} provides the wall-clock time of the proposed DDM for~\eqref{LLT} with respect to various $\N$.
The case $\N = 1$ represents CP described above.
We used the stop criterion~\eqref{stop} with $\mathrm{TOL} = 10^{-3}$.
We see that the wall-clock time is effectively reduced if $\N$ is large.
Figs.~\ref{Fig:LLT}(b),~(e) and~(c),~(f) shows the image results of the cases $\N = 2 \times 2$ and $\N = 16 \times 16$, respectively.
We observe that they show no trace on the subdomain interfaces.

\begin{remark}
\label{Rem:nonlocal}
The proposed DDM cannot be used for the problems with nonlocal structures such as nonlocal total variation minimization~\cite{ZC:2010}, since the essential domain on each subdomain $\Omega_s$ becomes the whole domain $\Omega$.
\end{remark}

\section{Conclusion}
\label{Sec:Conclusion}
In this paper, we proposed the overlapping domain decomposition framework for variational imaging problems using the notion of essential domains.
Due to the parallel structure of Lagrange multipliers, it is easy to implement the proposed DDM on distributed memory computers.
The proposed DDM was applied to various problems on image processing and showed superior performances compared to existing methods.
To the best of our knowledge, the proposed method is the first DDM which also accommodate higher order models.

This paper gives several subjects for future researches.
Since the convergence rate of the proposed DDM highly depends on a choice of a penalty parameter, it is worth to investigate a good way to choose the penalty parameter.
Recently, an acceleration technique for the alternating direction method of multipliers was proposed in~\cite{Kim:2019} and an application of this technique to the proposed DDM may yield a remarkable improvement of the convergence rate.
Finally, we expect that the proposed method is applicable to nonconvex problems with a little modification since several augmented Lagrangian approaches have been successfully applied to nonconvex problems~\cite{HLR:2016,WYZ:2019} recently.

 \section*{Conflict of interest}
 The authors declare that they have no conflict of interest.

\appendix
\section{Convergence analysis of Algorithm~\ref{Alg:DD}}
\label{App:analysis}
In this appendix, we analyze the convergence behavior of the decoupled augmented Lagrangian method.
Throughout this section, we assume that $\tE (\tu)$ given in~\eqref{DD} is convex.

The proof of Theorem~\ref{Thm:global} is based on a Lyapunov functional argument, which is broadly used in the analysis of augmented Lagrangian methods~\cite{HY:2015,HLR:2016,WT:2010}.
That is, we show that there exists the Lyapunov functional that is bounded below and decreases in each iteration.
The following lemma is a widely-used property for convex optimization.

\begin{lemma}
\label{Lem:convex}
Let $f$\emph{:}~$\mathbb{R}^n \rightarrow \bar{\mathbb{R}}$ be a convex function, $A$\emph{:}~$\mathbb{R}^n \rightarrow \mathbb{R}^m$ a linear operator, and $b \in \mathbb{R}^m$.
Then, a solution $x^* \in \mathbb{R}^n$ of the minimization problem
\begin{equation*}
\min_{x \in \mathbb{R}^n} f(x) + \frac{\alpha}{2} \| Ax - b \|_2^2
\end{equation*}
is characterized by
\begin{equation*}
f(x) \geq f(x^* ) + \alpha \langle Ax^* - b , A(x^* - x) \rangle \hspace{0.5cm} \forall x \in \mathbb{R}^n .
\end{equation*}
\end{lemma}
\begin{proof}
It is straightforward from the fact that $- \alpha A^* (Ax^* - b) \in \partial f(x^*)$.
\end{proof}

We observe that if we choose an initial guess $\lambda^{(0)} \in \tV^*$ such that $J_{\tV^*} \lambda^{(0)} \in (\ker B)^{\bot}$, then we have $J_{\tV^*} \lambda^{(n)} \in (\ker B)^{\bot}$ for all $n \geq 0$.

\begin{proposition}
\label{Prop:kerB}
In Algorithm~\ref{Alg:DD}, we have $J_{\tV^*} \lambda^{(n)} \in (\ker B)^{\bot}$ for all $n \geq 1$ if $J_{\tV^*} \lambda^{(0)} \in (\ker B)^{\bot}$.
\end{proposition}
\begin{proof}
Since $J_{\tV^*} ( \lambda^{(n+1)} - \lambda^{(n)} ) = (I -P_B ) \tu^{(n+1)} \in (\ker B)^{\bot}$ for all $n \geq 0$, a simple induction argument yields the conclusion.
\end{proof}

With Lemma~\ref{Lem:convex} and Proposition~\ref{Prop:kerB}, we readily get the following characterization of $\tu^{(n+1)}$ in Algorithm~\ref{Alg:DD}.

\begin{lemma}
\label{Lem:char}
In Algorithm~\ref{Alg:DD}, $\tu^{(n+1)} \in \tV$ satisfies
\begin{multline*}
\tE (\tu) \geq \tE (\tu^{(n+1)}) + \langle J_{\tV} (I - P_B) (\tu^{(n+1)} - \tu) , \lambda^{(n)} \rangle \\
+ \eta \langle \tu^{(n+1)} - P_B \tu^{(n)} , \tu^{(n+1)} - \tu \rangle \quad \forall \tu \in \tV
\end{multline*}
for $n \geq 0$.
\end{lemma}
\begin{proof}
Take any $\tu \in \tV$.
By~\eqref{tu_final}, Lemma~\ref{Lem:convex}, and Proposition~\ref{Prop:kerB}, we obtain
\begin{eqnarray*}
\tE (\tu) &\geq& \tE (\tu^{(n+1)}) + \langle \tu^{(n+1)} - \tu , J_{\tV^*} \lambda^{(n)} \rangle + \eta \langle \tu^{(n+1)} - P_B \tu^{(n)} , \tu^{(n+1)} - \tu \rangle \\
&=& \tE (\tu^{(n+1)}) + \langle \tu^{(n+1)} - \tu , (I-P_B ) J_{\tV^*} \lambda^{(n)} \rangle + \eta \langle \tu^{(n+1)} - P_B \tu^{(n)} , \tu^{(n+1)} - \tu \rangle,
\end{eqnarray*}
which concludes the proof.
\end{proof}

Let $(\tu^* , \lambda^*) \in \tV \times \tV^*$ be a critical point of~\eqref{saddle_pb}.
We define
\begin{subequations}
\label{de_n}
\begin{eqnarray}
\label{d_n}
d_n &=& \eta \| P_B (\tu^{(n)} - \tu^{(n+1)}) \|_{2}^2 + \frac{1}{\eta} \| \lambda^{(n)} - \lambda^{(n+1)} \|_{2}^2 , \\
\label{e_n}
e_n &=& \eta \| P_B (\tu^{(n)} - \tu^*) \|_{2}^2 + \frac{1}{\eta} \| \lambda^{(n)} - \lambda^* \|_{2}^2 .
\end{eqnarray}
\end{subequations}
It is clear that the value $d_n$ measures the difference between two consecutive iterates $(\tu^{(n)} , \lambda^{(n)})$ and $(\tu^{(n+1)}, \lambda^{(n+1)})$, while $e_n$ measures the error of the $n$th iterate $(\tu^{(n)} , \lambda^{(n)})$ with respect to a solution $(\tu^*, \lambda^*)$.
The following lemma presents the Lyapunov functional argument.
We note that the Lyapunov functional that we use in the proof is motivated from~\cite{WT:2010}. 

\begin{lemma}
\label{Lem:e_dec}
The value $e_n$ defined in~\eqref{e_n} is decreasing in each iteration of Algorithm~\ref{Alg:DD}.
More precisely, we have
\begin{equation}
\label{e_diff}
e_n - e_{n+1} \geq d_n
\end{equation}
for $n \geq 0$, where $d_n$ is given in~\eqref{d_n}.
\end{lemma}
\begin{proof}
By the definition of $(\tu^* , \lambda^*)$, we clearly have $(I - P_B) \tu^* = 0$.
Furthermore, since
\begin{equation*}
\tu^* \in \argmin_{\tu \in \tV} \left\{ \tE(\tu) + \langle J_{\tV} (I-P_B )\tu, \lambda^* \rangle + \frac{\eta}{2} \| (I -P_B ) \tu \|_{2}^2 \right\},
\end{equation*}
by Lemma~\ref{Lem:convex}, $\tu^*$ is characterized by
\begin{equation}
\label{u_char_temp}
\tE (\tu) \geq \tE (\tu^*) - \langle J_{\tV} (I-P_B) (\tu - \tu^*) , \lambda^* \rangle \hspace{0.5cm} \forall \tu \in \tV.
\end{equation}
Taking $\tu = \tu^{(n+1)}$ in \eqref{u_char_temp} yields
\begin{equation}
\label{u_char}
\tE (\tu^{(n+1)}) \geq \tE (\tu^*) - \langle J_{\tV} (I-P_B)\bu^{(n+1)} , \lambda^* \rangle.
\end{equation}
Let $\bu^{(n)} = \tu^{(n)} - \tu^*$ and $\bl^{(n)} = \lambda^{(n)} - \lambda^*$.
Taking $\tu = \tu^*$ in Lemma~\ref{Lem:char} yields
\begin{equation}
\label{un_char}
\tE (\tu^*) \geq \tE (\tu^{(n+1)}) + \langle J_{\tV} (I-P_B) \bu^{(n+1)} , \lambda^{(n)} \rangle + \eta \langle \tu^{(n+1)} - P_B \tu^{(n)} , \bu^{(n+1)}\rangle.
\end{equation}
Then, by adding \eqref{u_char} and \eqref{un_char} and using $P_B \tu^{(n)} = P_B \bu^{(n)}$, we have
\begin{eqnarray*}
0 &\geq& \langle J_{\tV} (I-P_B) \bu^{(n+1)} , \bl^{(n)} \rangle + \eta \langle \bu^{(n+1)} - P_B \bu^{(n)} , \bu^{(n+1)}\rangle \\
&=& \langle J_{\tV} (I-P_B) \bu^{(n+1)} , \bl^{(n)} \rangle + \eta \| (I-P_B) \bu^{(n+1)} \|_2^2 - \eta \langle P_B(\bu^{(n)} - \bu^{(n+1)}), \bu^{(n+1)}\rangle.
\end{eqnarray*}
That is, we obtain
\begin{multline}
\label{S}
S:= -\langle J_{\tV} (I-P_B) \bu^{(n+1)} , \bl^{(n)} \rangle - \eta \| (I-P_B) \bu^{(n+1)} \|_2^2 \\
+ \eta \langle P_B(\bu^{(n)} - \bu^{(n+1)}), \bu^{(n+1)}\rangle \geq 0.
\end{multline}
Using $\bl^{(n+1)} = \bl^{(n)} + \eta J_{\tV} (I-P_B ) \bu^{(n+1)}$, we get
\begin{equation*}
e_n - e_{n+1}
= 2S + \eta \| P_B(\bu^{(n)} - \bu^{(n+1)}) \|_2^2 + \eta \| (I-P_B) \bu^{(n+1)} \|_2^2
\,\,\geq d_n,
\end{equation*}
which yields \eqref{e_diff}.
The last inequality is due to~\eqref{S}.
\end{proof}

Now, we present the proof of Theorem~\ref{Thm:global}.
 
\begin{proof}[Proof of Theorem~\ref{Thm:global}]
As $\tE$ is convex, Lemma~\ref{Lem:e_dec} ensures that \eqref{e_diff} holds.
Since $\left\{e_n \right\}$ is bounded, we conclude that $\left\{ P_B \tu^{(n)} \right\}$ and $\left\{ \lambda^{(n)} \right\}$ are bounded. We sum~\eqref{e_diff} from $n=0$ to $N-1$ and let $N \rightarrow \infty$ to obtain
\begin{equation*}
e_0 - \lim_{n \rightarrow \infty} e_n \geq \sum_{n=0}^N d_n
= \eta \sum_{n=0}^{\infty} \| P_B (\tu^{(n)} - \tu^{(n+1)})\|_2^2 + \eta \sum_{n=0}^{\infty} \| (I-P_B) \tu^{(n+1)} \|_2^2 ,
\end{equation*}
which implies that $P_B(\tu^{(n)} - \tu^{(n+1)}) \rightarrow 0$ and $(I-P_B) \tu^{(n+1)} \rightarrow 0$.
Therefore, $\left\{ \tu^{(n)} \right\}$ is bounded and we have
\begin{subequations}
\label{diff_vanish}
\begin{equation}
\tu^{(n)} - \tu^{(n+1)} = P_B(\tu^{(n)} - \tu^{(n+1)}) + (I-P_B) \tu^{(n)} - (I-P_B) \tu^{(n+1)} \rightarrow 0
\end{equation}
and
\begin{equation}
\lambda^{(n)} - \lambda^{(n+1)} = -\eta J_{\tV} (I-P_B) \tu^{(n+1)} \rightarrow 0.
\end{equation}
\end{subequations}

By the Bolzano--Weierstrass theorem, there exists a limit point $(\tu^{(\infty)}, \lambda^{(\infty)})$ of the sequence $\left\{ (\tu^{(n)} , \lambda^{(n)}) \right\}$.
We choose a subsequence $\{ (\tu^{(n_j)} , \lambda^{(n_j)}) \}$ of $\left\{ (\tu^{(n)} , \lambda^{(n)}) \right\}$ such that
\begin{equation}
\label{subsequence}
(\tu^{(n_j)}, \lambda^{(n_j)}) \rightarrow (\tu^{(\infty)} , \lambda^{(\infty)}) \textrm{ as } j \rightarrow \infty.
\end{equation}
By~\eqref{diff_vanish}, we have 
\begin{equation*}
(\tu^{(n_j - 1)}, \lambda^{(n_j - 1)}) \rightarrow (\tu^{(\infty)} , \lambda^{(\infty)}) \textrm{ as } j \rightarrow \infty.
\end{equation*}
In the $\lambda$-update step with $n = n_j - 1$:
\begin{equation*}
\lambda^{(n_j)} = \lambda^{(n_j -1)} + \eta J_{\tV} (I - P_B) \tu^{(n_j) },
\end{equation*}
we readily obtain $(I -P_B) \tu^{(\infty)} = 0$ as $j$ tends to $\infty$.
On the other hand,~\eqref{tu_final} with $n = n_j - 1$ is equivalent to
\begin{equation*}
\partial \tE (\tu^{(n_j)}) + J_{\tV^*} \lambda^{(n_j - 1)} + \eta ( \tu^{(n_j )} - P_B \tu^{(n_j - 1)}) \ni 0.
\end{equation*}
By the graph-closedness of~$\partial \tE$~(see Theorem~24.4 in~\cite{Rockafellar:2015}), we get
\begin{equation*}
\partial \tE (\tu^{(\infty)} ) + J_{\tV^*} \lambda^{(\infty)} + \eta (I-P_B ) \tu^{(\infty)} \ni 0
\end{equation*}
as $j \rightarrow \infty$.
By Proposition~\ref{Prop:kerB}, we conclude that
\begin{equation*}
\partial \tE (\tu^{(\infty)} ) + (I-P_B ) J_{\tV^*} \lambda^{(\infty)} \ni 0.
\end{equation*}
Therefore, $(\tu^{(\infty)}, \lambda^{(\infty)})$ is a critical point of~\eqref{saddle_pb}.

Finally, it remains to prove that the whole sequence $\{ ( \tu^{(n)}, \lambda^{(n)} ) \}$ converges to the critical point $(\tu^{(\infty)}, \lambda^{(\infty)})$.
Since the critical point $(\tu^*, \lambda^*)$ was arbitrarily chosen, Lemma~\ref{Lem:e_dec} is still valid if we set $(\tu^*, \lambda^*) = (\tu^{(\infty)}, \lambda^{(\infty)})$ in~\eqref{e_n}.
That is, the sequence
\begin{equation*}
e_n = \eta \| P_B (\tu^{(n)} - \tu^{(\infty)}) \|_{2}^2 + \frac{1}{\eta} \| \lambda^{(n)} - \lambda^{(\infty)} \|_{2}^2
\end{equation*}
is decreasing.
On the other hand, by~\eqref{subsequence}, the subsequence $\{ e_{n_j} \}$ tends to $0$ as $j$ goes to $\infty$.
Therefore, the whole sequence $\{ e_n \}$ tends to $0$ and we deduce that $\{ ( \tu^{(n)}, \lambda^{(n)} ) \}$ converges to $(\tu^{(\infty)}, \lambda^{(\infty)})$.
\end{proof}

\begin{remark}
\label{Rem:inexact}
In practice, local problems~\eqref{local2} are solved by iterative algorithms and an inexact solution $\tu^{(n+1)}$ to~\eqref{tu_final} is obtained in each iteration of Algorithm~\ref{Alg:DD}.
That is, for $n \geq 0$, we have
\begin{equation*}
0 \in \partial_{\epsilon_n} J_n (\tu^{(n)})
\end{equation*}
for some $\epsilon_n > 0$, where
\begin{equation*}
J_n (\tu) = \tE (\tu) + \langle J_{\tV} \tu, \lambda^{(n)} \rangle_{\tV^*} + \frac{\eta}{2} \| \tu - P_B \tu^{(n)} \|_{2, \tV}^2.
\end{equation*}
One may refer, e.g.,~\cite{Rockafellar:2015} for the definition of the $\epsilon$-subgradient $\partial_{\epsilon}$.
In this case, the conclusion of Lemma~\ref{Lem:char} is replaced by
\begin{multline}
\label{char_inexact}
\tE (\tu) \geq \tE (\tu^{(n+1)}) + \langle J_{\tV} (I - P_B) (\tu^{(n+1)} - \tu) , \lambda^{(n)} \rangle \\
+ \eta \langle \tu^{(n+1)} - P_B \tu^{(n)} , \tu^{(n+1)} - \tu \rangle - \epsilon_n \quad \forall \tu \in \tV
\end{multline}
for all $n \geq 0$.
By slightly modifying the above proofs using~\eqref{char_inexact}, one can prove without major difficulty that the conclusion of Theorem~\ref{Thm:global} holds under an assumption
\begin{equation*}
\sum_{n=0}^{\infty} \epsilon_n < \infty.
\end{equation*}
The above summability condition of errors is popular in the field of mathematical optimization; see, e.g.,~\cite{Rockafellar:1976}.
\end{remark}

To prove Theorem~\ref{Thm:rate}, we first show that $d_n$ is decreasing.

\begin{lemma}
\label{Lem:d_dec}
The value $d_n$ defined in~\eqref{d_n} is decreasing in each iteration of Algorithm~\ref{Alg:DD}.
\end{lemma}
\begin{proof}
Let $n \geq 1$.
Taking $\tu = \tu^{(n)} $ in Lemma~\ref{Lem:char} yields
\begin{equation}
\label{d_dec1}
\tE( \tu^{(n)}) \geq \tE (\tu^{(n+1)}) + \langle J_{\tV} (I-P_B)(\tu^{(n+1)} - \tu^{(n)}) , \lambda^{(n)} \rangle + \eta \langle\tu^{(n+1)} - P_B \tu^{(n)} , \tu^{(n+1)} - \tu^{(n)} \rangle.
\end{equation}
Also, substituting $n$ by $n-1$ and taking $\tu = \tu^{(n+1)}$ in Lemma~\ref{Lem:char}, we have
\begin{equation}
\tE( \tu^{(n+1)}) \geq \tE (\tu^{(n)}) + \langle J_{\tV} (I-P_B)(\tu^{(n)} - \tu^{(n+1)}) , \lambda^{(n-1)} \rangle + \eta \langle\tu^{(n)} - P_B \tu^{(n-1)} , \tu^{(n)} - \tu^{(n+1)} \rangle.
\label{d_dec2}
\end{equation}
Summation of~\eqref{d_dec1} and~\eqref{d_dec2} yields
\begin{eqnarray*}
0 &\geq& \langle J_{\tV} (I-P_B ) (\tu^{(n)} - \tu^{(n+1)}) , \lambda^{(n-1)}-\lambda^{(n)} \rangle \\
&&+ \eta \langle \tu^{(n)} - \tu^{(n+1)}, \tu^{(n)} - P_B \tu^{(n-1)} - \tu^{(n+1)} + P_B \tu^{(n)} \rangle \\ 
&=& \eta \langle \tu^{(n)} - \tu^{(n+1)}, -P_B \tu^{(n-1)} + 2P_B \tu^{(n)} - \tu^{(n+1)} \rangle,
\end{eqnarray*}
where we used $\lambda^{(n)} = \lambda^{(n-1)} + \eta J_{\tV} (I-P_B) \tu^{(n)}$ in the equality.
Therefore, we get
\begin{equation}
\label{T}
T:= \langle \tu^{(n)} - \tu^{(n+1)}, P_B \tu^{(n-1)} - 2P_B \tu^{(n)} + \tu^{(n+1)} \rangle \geq 0.
\end{equation}
On the other hand, direct computation yields
\begin{equation*}
\frac{1}{\eta} (d_{n-1} - d_n ) 
= 2T + \| \tu^{(n)} - \tu^{(n+1)} - P_B (\tu^{(n-1)} - \tu^{(n)}) \|_2^2
\,\,\geq 0,
\end{equation*}
which concludes the proof.
The last inequality is due to~\eqref{T}.
\end{proof}

Combining Lemmas~\ref{Lem:e_dec} and~\ref{Lem:d_dec}, we get the proof Theorem~\ref{Thm:rate}, which closely follows~\cite{HY:2015}.

\begin{proof}[Proof of Theorem~\ref{Thm:rate}]
Invoking Lemmas~\ref{Lem:d_dec} and~\eqref{e_diff} yields
\begin{equation*}
(n+1)d_n \leq \sum_{k=0}^{n} d_k \leq e_0 - e_{n+1} \leq e_0.
\end{equation*}
This completes the proof.
\end{proof}

\section{A remark on the continuous setting}
\label{App:DCT}
As we noticed in Section~\ref{Sec:Applications}, the proposed domain decomposition framework reduces to the one proposed in~\cite{DCT:2016} when it is applied to the convex Chan--Vese model~\cite{CE:2005}.
However, while the authors of~\cite{DCT:2016} introduced their method as a nonoverlapping DDM, we classified it as an overlapping one.
In this section, we claim that the proposed method belongs to a class of overlapping DDMs in the continuous setting.

For simplicity, we consider the case $\N = 2$ only.
Let $\left\{ \Omega_s \right\}_{s=1}^2$ be a nonoverlapping domain decomposition of $\Omega$ with the interface $\Gamma = \partial \Omega_1 \cap \partial \Omega_2$.
Recall the convex Chan--Vese model~\eqref{CCV}:
\begin{equation}
\label{CCV_re}
\min_{u \in BV(\Omega)} \left\{ \alpha \intO ug \,dx + \chi_{\left\{ 0 \leq \cdot \leq 1\right\}} (u) + TV_{\Omega}(u) \right\},
\end{equation}
where $g = (f - c_1 )^2 - (f-c_2 )^2$ and $TV_{\Omega}(u)$ is defined as
\begin{equation*}
TV_{\Omega} (u) = \sup \left\{\intO u \div \p \,dx : \p \in C_0^1(\Omega, \mathbb{R}^2), |\p| \leq 1 \right\}.
\end{equation*}
In Section~3.1 of~\cite{DCT:2016}, it was claimed that a solution of~\eqref{CCV_re} can be constructed by $u = u_1 \oplus u_2$, where $(u_1, u_2)$ is a solution of the constrained minimization problem
\begin{equation}
\label{CCV_DCT}
\min_{\substack{u_s \in BV(\Omega_s) \\ s=1,2}} \sum_{s=1}^2 \left( \alpha \intOs u_s g \,dx + \chi_{\left\{ 0 \leq \cdot \leq 1 \right\} } (u_s) + TV_{\Omega_s}(u_s) \right) \hspace{0.3cm} \textrm{subject to } u_1 = u_2 \textrm{ on } \Gamma. 
\end{equation}
Here, the condition $u_1 = u_2$ on $\Gamma$ is of the trace sense~\cite{EG:1992}.
Unfortunately, this argument is not valid since the solution space $BV(\Omega)$ of~\eqref{CCV_re} allows discontinuities on $\Gamma$.
We provide a simple counterexample inspired from~\cite{CE:2005}.

\begin{example}
Let $\Omega = (-1, 1) \subset \mathbb{R}$, $\Omega_1 = (-1, 0)$, and $\Omega_2 = (0, 1)$.
We set
$$
g(x) = \begin{cases} -1 & \textrm{ if } x \in \Omega_1 , \\ 1 & \textrm{ if } x \in \Omega_2 . \end{cases}
$$
We will show that
$$
u^*(x) = \begin{cases} 1 & \textrm{ if } x \in \Omega_1 , \\ 0 & \textrm{ if } x \in \Omega_2 \end{cases}
$$
is a unique solution of~\eqref{CCV_re} for sufficiently large $\alpha$, while it cannot be a solution of~\eqref{CCV_DCT} since it is not continuous on $\Gamma$.
We clearly have $TV_{\Omega}(u^*) = 1$.
There exists $p^* \in C_0^1 (\Omega)$ with $|p^*| \leq 1$ which attains the supremum in the definition of total variation for $u^*$. Indeed, with $p^*(x) = 1-x^2$, we have
\begin{equation*}
1 = TV_{\Omega}(u^* ) = \sup\left\{ \intO u^* p' \,dx : p \in C_0^1 (\Omega), |p| \leq 1 \right\} \\
\geq \intO u^* (p^*)' \,dx = 1. 
\end{equation*}
Choose $\alpha > 2 = \max_{x \in \Omega} |(p^*)'(x)|$.
For any $u \in BV(\Omega)$ with $0 \leq u \leq 1$, we have
\begin{eqnarray*}
\alpha \intO ug \,dx + TV_{\Omega}(u) &\geq& \alpha \intO ug \,dx + \intO u(p^*)' \,dx \\
&=& \alpha \intO u^* g \,dx + TV_{\Omega}(u^*) + \intO (u-u^*) (\alpha g + (p^*)') \,dx. \\
\end{eqnarray*}
In addition, we have
\begin{equation*}
\intO (u-u^*) (\alpha g + (p^*)')\,dx = \int_{-1}^0 (1-u) (\alpha - (p^*)') \,dx + \int_0^1 u (\alpha + (p^*)')\,dx \geq 0. 
\end{equation*}
Since $\alpha \pm (p^*)'$ is strictly positive, the equality holds if and only if $u = u^*$ a.e..
Therefore, $u^*$ is a unique solution of~\eqref{CCV_re}.
\end{example}

On the other hand, it is possible to construct an equivalent constrained minimization problem with an overlapping domain decomposition.
Let $S$ be a neighborhood of $\Gamma$ with positive measure.
Note that traces $\gamma_1 u$ and $\gamma_2 u$ of $u \in BV(\mathcal{O})$ along $\Gamma$ with respect to $\mathcal{O} \cap \Omega_1$ and $\mathcal{O} \cap \Omega_2$, respectively, are well-defined for any open subset $\mathcal{O}$ of $\Omega$ such that $S \subset \mathcal{O}$.
Also, they satisfy the formula
\begin{equation*}
TV_{\Omega} (u) = TV_{\Omega_1} (u) + TV_{\Omega_2} (u) + \int_{\Gamma} |\gamma_1 u - \gamma_2 u | \,ds .
\end{equation*}
Set $\tOmega_1 = \Omega_1 \cup S$ and $\tOmega_2 = \Omega_2$.
Then, $\{ \tOmega_s \}_{s=1}^2$ forms an overlapping domain decomposition of $\Omega$, i.e., $\tGamma = \tOmega_1 \cap \tOmega_2$ has positive measure.
We define local energy functionals $E_s$:~$\tOmega_s \rightarrow \mathbb{R}$ as follows:
\begin{eqnarray*}
E_1 (\tu_1) &=& \alpha \int_{\Omega_1} \tu_1 g \,dx + \chi_{\left\{ 0 \leq \cdot \leq 1 \right\}} (\tu_1) + TV_{\Omega_1} (\tu_1) + \int_{\Gamma} |\gamma_1 \tu_1 - \gamma_2 \tu_1 | \,ds , \\
E_2 (\tu_2) &=& \alpha \int_{\Omega_2} \tu_2 g \,dx + \chi_{\left\{ 0 \leq \cdot \leq 1 \right\}} (\tu_2) + TV_{\Omega_2} (\tu_2).
\end{eqnarray*}
Consider the following constrained minimization problem:
\begin{equation}
\label{CCV_DD}
\min_{\substack{\tu_s \in BV(\tOmega_s) \\ s=1,2}} \sum_{s=1}^2 E_s (\tu_s) \hspace{0.3cm} \textrm{subject to } \tu_1 = \tu_2 \textrm{ on } \tGamma .
\end{equation}
Then, we have the following equivalence theorem.

\begin{theorem}
\label{Thm:cont}
Let $(\tu_1^*, \tu_2^* ) \in BV(\tOmega_1) \times BV(\tOmega_2)$ be a solution of~\eqref{CCV_DD}.
Then, $u^* \in BV(\Omega)$ defined by
$$
u^*(x) = \begin{cases} \tu_1^* (x) & \textrm{ if } x \in \tOmega_1 , \\ \tu_2^* (x) & \textrm{ if } x \in \Omega \setminus \tOmega_1 \end{cases}
$$
is a solution of~\eqref{CCV_re}.
Conversely, if $u^* \in BV(\Omega)$ is a solution of~\eqref{CCV_re}, then $(\tu_1^* , \tu_2^*) = (u^*|_{\tOmega_1} , u^*|_{\tOmega_2}) \in BV(\tOmega_1) \times BV(\tOmega_2)$ is a solution of~\eqref{CCV_DD}.
\end{theorem}
\begin{proof}
First, suppose that $(\tu_1^*, \tu_2^* )$ is a solution of~\eqref{CCV_DD}.
For any $u \in BV(\Omega)$ with $0 \leq u \leq 1$, we have
\begin{equation*} \begin{split}
\alpha \intO ug \,dx + TV_{\Omega}(u) 
&= E_1 (u|_{\tOmega_1}) + E_2 (u|_{\tOmega_2}) \\
&\geq E_1 (\tu_1^*) + E_2 (\tu_2^*) 
\,\,= \alpha \intO u^* g \,dx + TV_{\Omega}(u^*).
\end{split} \end{equation*}
Hence, $u^*$ minimizes~\eqref{CCV_re}.

Conversely, we assume that $u^* \in BV(\Omega)$ is a solution of~\eqref{CCV_re} and set $(\tu_1^* , \tu_2^*) = (u^*|_{\tOmega_1} , u^*|_{\tOmega_2})$.
Take any $(\tu_1 , \tu_2) \in BV(\tOmega_1) \times BV(\tOmega_2)$ such that $0 \leq \tu_1 \leq 1$, $0 \leq \tu_2 \leq 1$, and $\tu_1 = \tu_2$ on $\tGamma$.
Let
$$
u(x) = \begin{cases} \tu_1 (x) & \textrm{ if } x \in \tOmega_1 , \\ \tu_2 (x) & \textrm{ if } x \in \Omega \setminus \tOmega_1 . \end{cases}
$$
Then we have
\begin{equation*} \begin{split}
E_1 (\tu_1) + E_2 (\tu_2) 
&= \alpha \intO ug \,dx + TV_{\Omega} (u) \\
&\geq \alpha \intO u^* g \,dx + TV_{\Omega} (u^*)
\,\,= E_1 (\tu_1^*) + E_2 (\tu_2^*).
\end{split} \end{equation*}
Therefore, $(\tu_1^* , \tu_2^*)$ is a solution of~\eqref{CCV_DD}.
\end{proof}

In conclusion, it is more appropriate to classify the proposed DDM in~\cite{DCT:2016} as an overlapping one instead of a nonoverlapping one.



\end{document}